\documentclass[10pt]{amsart}
\usepackage{amssymb}
\usepackage{bm}
\usepackage{graphicx}
\usepackage[centertags]{amsmath}
\usepackage{amsmath}
\usepackage{amsfonts}
\usepackage{amsthm}
\usepackage{amsbsy}
\usepackage{mathtools}
\usepackage{mathrsfs}
\usepackage{cases}
\usepackage{xfrac}
\usepackage[all]{xy}
\usepackage[linktocpage=true]{hyperref}
\usepackage{hyperref}
\usepackage{dsfont}
\usepackage{bbm}
\usepackage{enumitem}
\hypersetup{
    colorlinks=true,
    linkcolor=blue,
    filecolor=blue,
    citecolor=blue,
    urlcolor=cyan}
\usepackage[margin=3.5cm]{geometry}
\newtheorem{thm}{Theorem}[section]
\newtheorem{cor}[thm]{Corollary}
\newtheorem{lem}[thm]{Lemma}

\newtheorem{prop}[thm]{Proposition}
\newtheorem{claim}[thm]{Claim}
\newtheorem{fact}[thm]{Fact}
\newtheorem*{DPHJ-conjecture}{Density polynomial Hales--Jewett conjecture}

\newtheorem{problem}[thm]{Problem}
\newtheorem{defn}[thm]{Definition}

\theoremstyle{remark}
\newtheorem{rem}[thm]{Remark}

\DeclarePairedDelimiterX\Set[1]\{\}{%

#1
}

\newcommand{\meg}{\geqslant}
\newcommand{\mik}{\leqslant}
\newcommand{\ave}{\mathbb{E}}
\newcommand{\prob}{\mathbb{P}}

\newcommand{\Ecal}{\mathcal{E}}

\newcommand{\Gcal}{\mathcal{G}}

\newcommand{\Ccal}{\mathcal{C}}

\newcommand{\Hcal}{\mathcal{H}}

\newcommand{\Ical}{\mathcal{I}}
\newcommand{\Xcal}{\mathcal{X}}
\newcommand{\Ftwo}[1]{{\mathbb{F}_2^{#1}}}

\DeclareMathOperator{\conj}{conj}

\begin{document}

\title[Uniformity of extremal graph-codes]{Uniformity of extremal graph-codes}

\author{No\'{e} de Rancourt, Pandelis Dodos and Konstantinos Tyros}

\address{Universit\'{e} de Lille, CNRS, UMR 8524, Laboratoire Paul Painlev\'{e}, F-59 000 Lille, France}
\email{nderancour@univ-lille.fr}

\address{Department of Mathematics, University of Athens, Panepistimiopolis 157 84, Athens, Greece}
\email{pdodos@math.uoa.gr}

\address{Department of Mathematics, University of Athens, Panepistimiopolis 157 84, Athens, Greece}
\email{ktyros@math.uoa.gr}



\begin{abstract}
It is an important fact that extremal discrete structures---that is, discrete structures of maximal size among those that avoid certain configurations---exhibit strong pseudorandom behavior. We present instances of this phenomenon in the context of \emph{graph-codes}, a notion put forth recently by Alon, as well as on problems related to the density polynomial Hales--Jewett conjecture.
\end{abstract}

\maketitle



\section{Introduction} \label{sec1}

\numberwithin{equation}{section}

Extremal discrete structures---that is, discrete structures of maximal size among those that avoid certain configurations---are known to exhibit, or conjectured to have, strong pseudorandomness properties. A well-known and beautiful conjecture in this direction, due to S\'{o}s \cite{Sos13}, predicts that, for all $k$ sufficiently large, any graph on $\mathrm{R}(k)-1$ vertices\footnote{Here, for any integer $k\meg 2$, $\mathrm{R}(k)$ denotes the corresponding diagonal Ramsey number, that is, the least positive integer $n$ such that any two-coloring of the edges of the complete graph on $n$ vertices has a monochromatic clique of size $k$. (See, also, Theorem \ref{thm-Ramsey-main} in the main text.)} with no clique or independent set of order $k$, is necessarily quasirandom in the sense of Chung--Graham--Wilson \cite{CGW88,CGW89}.

The goal of this paper is to unravel similar phenomena in the context of the emerging field of \emph{graph-codes}---a notion put forth, very recently, by Alon \cite{Alon24}---as well as on problems related to the \emph{density polynomial Hales--Jewett conjecture}, a fundamental conjecture in Ramsey theory due to Bergelson \cite{Ber96}.

\subsection{Graph-codes, $\mathrm{HJ}$-codes, and related problems} \label{subsec1.1}

To put things in a proper context, we begin by introducing some pieces of notation and terminology. For every integer $n\meg 1$, we set $[n]:=\{1,\dots,n\}$; by convention, we set $[0]:=\emptyset$. For every finite set $V$ and every integer $r\meg 1$, by $\binom{V}{r}$ we denote the set of all subsets of $V$ with exactly $r$ elements, and by $\binom{V}{\leqslant r}$ the set of all \emph{nonempty} subsets of $V$ with at most $r$ elements; that~is,
\begin{equation} \label{e1.1}
\binom{V}{\leqslant r} := \binom{V}{1}\cup \dots \cup\binom{V}{r}.
\end{equation}

\subsubsection{Graphs} \label{subsubsec1.1.1}

All graphs in this paper have no multiple edges, but we allow the existence of self-loops. Thus, given a nonempty vertex set $V$, a graph $G$ on $V$ is a subset of $\binom{V}{\mik 2}$; we shall refer to the elements of $G\cap \binom{V}{2}$ as its \textit{edges}, and the elements of $G \cap \binom{V}{1}$ as its \textit{self-loops}. The \textit{spanning vertex set} $V(G)$ of $G$ is defined by setting
\begin{equation} \label{e1.2}
V(G) := \bigcup_{e \in G} e.
\end{equation}
Two graphs $G$ and $H$ are said to be \emph{isomorphic} if there exists a bijection $\phi \colon V(G) \to V(H)$ such that
\[ \{x, y\} \in G \ \ \ \text{ if and only if } \ \ \ \big\{\phi(x), \phi(y)\big\} \in H, \]
for all $x, y \in V(G)$; in particular, for all $x \in V(G)$, we have $\{x\} \in G$ if and only if $\{\phi(x)\} \in H$.

\subsubsection{Symmetric difference of graphs} \label{subsubsec1.1.2}

If $G_1, G_2$ are graphs on the same vertex set~$V$, then~let
\begin{equation} \label{e1.3}
G_1 + G_2 := (G_1 \setminus G_2) \cup (G_2 \setminus G_1)
\end{equation}
denote their \textit{symmetric difference}. If, in addition, we have $G_1\supseteq G_2$, then observe that $G_1+G_2$ coincides with the \emph{difference} $G_1 \setminus G_2$ of $G_1$ and $G_2$.

\subsubsection{Spaces of graphs} \label{subsubsec1.1.3}

For every positive integer $n$, we naturally identify, via indicator functions, loopless graphs on $[n]$ with elements of the abelian group $\mathbb{F}_2^{\binom{[n]}{2}}$; respectively, we identify graphs on $[n]$ that are not necessarily loopless with elements of $\mathbb{F}_2^{\binom{[n]}{\leqslant 2}}$. Note that, with these identifications, the operation of symmetric difference between graphs introduced in~\eqref{e1.3} corresponds to addition in the vector spaces $\mathbb{F}_2^{\binom{[n]}{2}}$ and $\mathbb{F}_2^{\binom{[n]}{\leqslant 2}}$; this observation actually justifies our notation in \eqref{e1.3}.

If $\Ical$ is a nonempty finite set---such as, $\binom{[n]}{2}$ and $\binom{[n]}{\mik 2}$---then by $\prob$ we denote the uniform probability measure on $\mathbb{F}_2^{\Ical}$. That is, for every $\Gcal\subseteq \mathbb{F}_2^{\Ical}$,
\begin{equation} \label{e1.4}
\prob[\Gcal]:= \frac{|\Gcal|}{2^{|\Ical|}};
\end{equation}
the probability $\prob[\Gcal]$ will often be referred to as the \textit{density} of $\Gcal$. The expectation of a function $f \colon \mathbb{F}_2^{\Ical} \to \mathbb{C}$ relative to $\prob$ will be denoted by $\mathbb{E}_{x \in \mathbb{F}_2^{\Ical}}[f(x)]$, or simply by $\mathbb{E}[f]$ if the index set $\Ical$ is understood from the context.

\subsubsection{Graph-codes} \label{subsubsec1.1.4}

We are now in a position to recall the notion of a graph-code.

\begin{defn}[$\mathcal{H}$-codes and $\mathcal{H}$-$\mathrm{HJ}$-codes] \label{d1.1}
Fix a collection $\Hcal$ of nonempty graphs, and let\, $\mathcal{G}$ be a collection of graphs on the same vertex set $V$.
\begin{enumerate}
\item[(i)] We say that $\mathcal{G}$ is an \emph{$\mathcal{H}$-code} if for every $G_1, G_2\in \mathcal{G}$, the graph $G_1 + G_2$ is not isomorphic to a graph in $\mathcal{H}$.
\item[(ii)] Respectively, we say that $\mathcal{G}$ is an \emph{$\mathcal{H}$-$\mathrm{HJ}$-code}\footnote{The acronym ``HJ" comes from ``Hales--Jewett".} if for every $G_1, G_2\in \mathcal{G}$ with $G_1\supseteq G_2$, the graph $G_1 + G_2$ is not isomorphic to a graph in $\mathcal{H}$.
\end{enumerate}
If\, $\mathcal{H}=\{H\}$ is a singleton, then we simply say that\, $\mathcal{G}$ is an \emph{$H$-code} or an \emph{$H$-$\mathrm{HJ}$-code}, respectively.
\end{defn}

\begin{rem} \label{r1.2}
Graph-codes were introduced by Alon \cite{Alon24}, but they have also appeared earlier in the comments section of Gowers' blog post \cite{Gow09}; we refer to~\cite{Alon24} for a detailed discussion on the history of this notion and its connections  with classical problems in extremal combinatorics. $\mathrm{HJ}$-codes have not been formally introduced in the literature (in particular, our terminology is not standard), but they have also appeared in the discussion in \cite{Gow09}. As it was pointed out in \cite{Gow09}, and we shall also explicitly see in Appendix \ref{sec-appendix}, $\mathrm{HJ}$-codes are closely related to the first unknown case of the density polynomial Hales--Jewett conjecture.
\end{rem}

There are several basic problems for graph-codes and $\mathrm{HJ}$-codes, and we shall recall those which are more relevant to the contents of this paper.

For any positive integer $r$, set
\begin{equation} \label{e1.5}
K_r:=\binom{[r]}{2} \ \ \ \text{ and } \ \ \ K_r^\circ:=\binom{[r]}{\mik 2};
\end{equation}
namely, $K_r$ denotes the complete loopless graph on $r$ vertices, and $K_r^\circ$ denotes the complete graph on $r$ vertices with all possible self-loops. Notice that $K_1$ is empty, but $K_1^\circ$ is nonempty.

\begin{problem}[Alon/Gowers] \label{p1.3}
If\, $\mathcal{G} \subseteq \mathbb{F}_2^{\binom{[n]}{2}}$ is a $K_4$-code, then\footnote{Here, $o_{n\to\infty}(1)$ denotes a quantity that goes to zero as $n$ tends to infinity. More generally, given a collection $C_1,\dots, C_r$ of parameters, $o_{C_1,\dots,C_r;n\to\infty}(1)$ denotes a quantity that goes to zero as $n$ tends to infinity with a rate of convergence that depends on the parameters $C_1,\dots, C_r$.} $\prob[\Gcal]=o_{n\to\infty}(1)$.
\end{problem}

Problem \ref{p1.3} is, arguably, the simplest open problem concerning graph-codes. We also have the following far reaching extension of Problem \ref{p1.3}.

\begin{problem}[Alon] \label{p1.4}
Let $H$ be a nonempty loopless graph with an even number of edges. If\, $\mathcal{G} \subseteq \mathbb{F}_2^{\binom{[n]}{2}}$ is an $H$-code, then\, $\prob[\Gcal]=o_{H;n\to\infty}(1)$.
\end{problem}

\begin{rem} \label{r1.5}
If $W$ is any loopless graph with an \emph{odd} number of edges, then the family $\Ecal\subseteq \mathbb{F}_2^{\binom{[n]}{2}}$ of all loopless graphs with an even number of edges, is a $W$-code and it satisfies $\prob[\Ecal]=\frac12$. Thus, an affirmative answer to Problem \ref{p1.4} will identify the parity of the number of edges of $H$ as the only obstruction to the existence of a dense $H$-code.
\end{rem}

The next three problems are related to graph-codes with respect to cliques. More precisely,~let
\begin{equation} \label{e1.6}
\Ccal:=\big\{ K_r\colon r\meg 2\big\} \ \ \ \text{ and } \ \ \
\Ccal^\circ:= \big\{ K_r^\circ\colon r\meg 1\big\}
\end{equation}
denote the families of loopless cliques, and cliques with all possible self-loops, respectively.

\begin{problem}[Alon/Gowers] \label{p1.6}
If\, $\mathcal{G} \subseteq \mathbb{F}_2^{\binom{[n]}{2}}$ is a\, $\Ccal$-code, then\, $\prob[\Gcal]=o_{n\to\infty}(1)$.
\end{problem}

\begin{problem}[Gowers] \label{p1.7}
If\, $\mathcal{G} \subseteq \mathbb{F}_2^{\binom{[n]}{2}}$ is a\, $\Ccal$-$\mathrm{HJ}$-code, then\, $\prob[\Gcal]=o_{n\to\infty}(1)$.
\end{problem}

\begin{problem} \label{p1.8}
If\, $\mathcal{G} \subseteq \mathbb{F}_2^{\binom{[n]}{\mik 2}}$ is a\, $\Ccal^\circ$-$\mathrm{HJ}$-code, then\, $\prob[\Gcal]=o_{n\to\infty}(1)$.
\end{problem}

Problems \ref{p1.7} and \ref{p1.8} are essentially equivalent in nature, the only difference being that Problem \ref{p1.7} refers to loopless graphs while Problem \ref{p1.8} to graphs with self-loops. It is clear that an affirmative answer to Problem \ref{p1.7} yields an affirmative answer to Problem \ref{p1.6}, and it is not hard to see\footnote{Indeed, fix a $\Ccal$-$\mathrm{HJ}$-code $\Gcal\subseteq \mathbb{F}_2^{\binom{[n]}{2}}$, and for every $G\in\Gcal$ and every $\emptyset\neq X\subseteq [n]$ with even cardinality, set $G_X:=G\cup \big\{ \{i\}\colon i\in X\big\}\in \mathbb{F}_2^{\binom{[n]}{\mik 2}}$. Then, the family $\Gcal':=\{G_X\colon G\in\Gcal \text{ and } \emptyset\neq X\subseteq [n] \text{ with even cardinality}\}$ is a $\Ccal^\circ$-$\mathrm{HJ}$-code with $\prob[\Gcal']\asymp \prob[\Gcal]$.} that an affirmative answer to Problem \ref{p1.8} yields an affirmative answer to Problem \ref{p1.7}. Finally, as we shall explain in Appendix \ref{sec-appendix}, Problem \ref{p1.8} is equivalent to the first unknown case of the density polynomial Hales--Jewett conjecture. Thus, Problems~\ref{p1.6}--\ref{p1.8} form a tower of critical test cases, of increasing difficulty, towards a resolution of the density polynomial Hales--Jewett conjecture. That said, we emphasize that all these problems are quite hard, and an affirmative resolution of any of these would constitute a significant advance.

\subsection{Gowers uniformity norms, and pseudorandomness} \label{subsec1.2}

Pseudorandomness refers to the phenomenon that certain deterministic (and explicit) discrete structures behave like random ones for most practical purposes. The phenomenon was first discovered in the context of graphs by Chung--Graham--Wilson \cite{CGW88,CGW89} who built upon previous work of Thomason \cite{Tho87}. Over the last thirty years, pseudorandomness has become one of the central themes of extremal and probabilistic combinatorics, and it has found numerous applications in number theory and theoretical computer science; see, e.g., \cite{Ro15}.

Much of the modern theory of pseudorandomness is developed using the uniformity norms introduced by Gowers \cite{Gow01}. We recall their definition in characteristic two, and refer the reader to \cite{TV06} for a more comprehensive treatment.

\begin{defn}[Uniformity norms in characteristic two] \label{d1.9}
Let $\Ical$ be a nonempty finite set, and let $d \geqslant 2$ be an integer. The \emph{(Gowers) uniformity norm $\|\cdot\|_{U_d}$ of order $d$} is defined by setting, for any $f \colon \mathbb{F}_2^{\Ical} \to \mathbb{C}$,
\begin{equation} \label{e1.7}
\|f\|_{U_d} := \bigg| \underset{x, y_1, \ldots, y_d \in \mathbb{F}_2^{\Ical}}{\ave} \bigg[
\prod_{s \subseteq [d]} \conj^{|s|}f\Big(x + \sum_{i \in s} y_i\Big)\bigg] \bigg|^{\frac{1}{2^d}},
\end{equation}
where $\conj$ denotes complex conjugation, that is, $\conj z := \bar z$.
\end{defn}

It is well-known---see \cite[(11.7)]{TV06}---that $\langle \|\cdot\|_{U_d}\colon d \meg 2\rangle$ is an increasing family of norms. The $U_2$-norm has a spectral interpretation that makes it amenable to discrete Fourier analysis (see Fact \ref{prop:U2ByFourier}). The $U_3$-norm is somewhat more involved, but it is still rather easy to grasp, and quantitatively effective to work with, thanks to the recent resolution of Marton's conjecture; see, in particular, \cite[Corollary 1.6]{GGMT25}. However, for $d\meg 4$, the behavior of the $U_d$-norm is significantly more complicated and it is understood via the deep work of Tao--Ziegler \cite{TZ12} on the inverse theorem for the $U_d$-norms (see Theorem \ref{thm:inverse}).

The link between pseudorandomness and the Gowers uniformity norms is provided by the notion of uniformity: a function $f \colon \mathbb{F}_2^{\Ical} \to \mathbb{C}$ is said to be \emph{$\varepsilon$-uniform of order $d$}, where $\varepsilon$ is a (small) positive parameter, or simply \emph{$d$-uniform}, if $\big\|f-\ave[f]\big\|_{U_d}\mik \varepsilon$; in particular, a~set $\Gcal\subseteq \mathbb{F}_2^{\Ical}$ is $d$-uniform if $\big\|\mathbbm{1}_{\Gcal}-\prob[\Gcal]\big\|_{U_d}=o(1)$. As $d$ becomes larger, $d$-uniformity becomes a stronger and substantially more informative property; we refer to \cite{HHL19} for an overview of these pseudorandomness properties and their applications in theoretical computer science.

\subsection{Main results} \label{subsec1.3}

As already noted, the goal of this paper is to show the pseudorandomness of graph-codes that are extremal in the sense of the following definition.

\begin{defn}[Extremal graph-codes] \label{d1.10}
Let $\Hcal$ be a collection of nonempty graphs, and let $n\meg 2$ be an integer. We set
\begin{gather}
\label{e1.8} \delta_n(\Hcal) :=
\max\bigg\{ \prob[\Gcal]\colon \Gcal\subseteq \mathbb{F}_2^{\binom{[n]}{2}} \text{ is an $\Hcal$-code}\bigg\}, \\
\label{e1.9} \Delta_n(\Hcal) :=
\max\bigg\{ \prob[\Gcal]\colon \Gcal\subseteq \mathbb{F}_2^{\binom{[n]}{2}} \text{ is an $\Hcal$-$\mathrm{HJ}$-code}\bigg\};
\end{gather}
observe that $\delta_n(\Hcal)\mik \Delta_n(\Hcal)$. We say that an\, $\Hcal$-code $\Gcal\subseteq \mathbb{F}_2^{\binom{[n]}{2}}$ is \emph{extremal} if\, $\prob[\Gcal]=\delta_n(\Hcal)$. Respectively, we say that an\, $\Hcal$-$\mathrm{HJ}$-code $\Gcal\subseteq \mathbb{F}_2^{\binom{[n]}{2}}$ is \emph{extremal} if\, $\prob[\Gcal]=\Delta_n(\Hcal)$.

We have the following versions of these invariants for graphs with loops,
\begin{gather}
\label{e1.10} \delta^\circ_n(\Hcal) :=
\max\bigg\{ \prob[\Gcal]\colon \Gcal\subseteq \mathbb{F}_2^{\binom{[n]}{\mik 2}} \text{ is an $\Hcal$-code}\bigg\}, \\
\label{e1.11} \Delta^\circ_n(\Hcal) :=
\max\bigg\{ \prob[\Gcal]\colon \Gcal\subseteq \mathbb{F}_2^{\binom{[n]}{\mik 2}} \text{ is an $\Hcal$-$\mathrm{HJ}$-code}\bigg\};
\end{gather}
again, note that $\delta^\circ_n(\Hcal)\mik \Delta^\circ_n(\Hcal)$. We say that an $\Hcal$-code $\Gcal\subseteq \mathbb{F}_2^{\binom{[n]}{\mik 2}}$ is \emph{extremal} if\, $\prob[\Gcal]=\delta^\circ_n(\Hcal)$. Finally, we say that an $\Hcal$-$\mathrm{HJ}$-code $\Gcal\subseteq \mathbb{F}_2^{\binom{[n]}{\mik 2}}$ is \emph{extremal} if\, $\prob[\Gcal]=\Delta^\circ_n(\Hcal)$.
\end{defn}
We will need the following simple fact.
\begin{fact} \label{f1.11}
For any collection $\mathcal{H}$ of nonempty graphs, the sequences $\big(\delta_n(\Hcal)\big)$, $\big(\Delta_n(\Hcal)\big)$, $\big(\delta^\circ_n(\Hcal)\big)$ and $\big(\Delta^\circ_n(\Hcal)\big)$ are all non-increasing.
\end{fact}
\begin{proof}
We will only argue for the sequence $\big(\delta_n(\Hcal)\big)$; the other cases are similar. Fix a pair $n>m\meg 2$ of integers, and let $\Gcal\subseteq \mathbb{F}_2^{\binom{[n]}{2}}$ be an extremal $\mathcal{H}$-code, that is, $\prob[\Gcal]=\delta_n(\Hcal)$. Set $\Xcal := \mathbb{F}_2^{\binom{[n]}{2}\setminus\binom{[m]}{2}}$ and, for every $x\in\mathcal{X}$, let $\Gcal_x:=\Big\{ W\in \mathbb{F}_2^{\binom{[m]}{2}}\colon (W,x)\in\Gcal\Big\}$ denote the section of $\Gcal$ at $x$. Since $\mathcal{G}$ is an $\mathcal{H}$-code, for every $x\in\mathcal{X}$, the family $\Gcal_x\subseteq \mathbb{F}_2^{\binom{[m]}{2}}$ is also an $\mathcal{H}$-code and, consequently, $\prob[\Gcal_x]\mik \delta_m(\Hcal)$. Thus,
\begin{equation} \label{e1.12}
\delta_n(\Hcal) = \prob[\Gcal] = \underset{x\in\Xcal}{\ave}\big[ \prob[\Gcal_x]\big] \leqslant \delta_m(\Hcal). \qedhere
\end{equation}
\end{proof}

We are now ready to state the first main result of this paper.

\begin{thm}[Fourier uniformity of extremal graph-codes] \label{thm:MainEvenEdges}
Let $\Hcal$ be a collection of nonempty loopless graphs, each with an even number of edges. Then, for every $\varepsilon>0$, there exists a positive integer $n_0=n_0(\varepsilon,\Hcal)$ such that, for all $n\geqslant n_0$, if\, $\mathcal{G}\subseteq \mathbb{F}_2^{\binom{[n]}{2}}$ is an extremal $\mathcal{H}$-code,~then
\begin{equation} \label{e1.13}
\big\|\mathbbm{1}_{\Gcal}-\prob[\Gcal]\big\|_{U_2} \mik \varepsilon.
\end{equation}
\end{thm}

\begin{rem} \label{r1.13}
Theorem \ref{thm:MainEvenEdges} implies in particular that if $H$ is any nonempty loopless graph with an even number of edges, then any extremal $H$-code is (asymptotically) Fourier uniform. On the other hand, as mentioned in Remark \ref{r1.5}, if $W$ is any loopless graph with an odd number of edges, then the family $\mathcal{E}\subseteq \mathbb{F}_2^{\binom{[n]}{2}}$ of all loopless graphs with an even number of edges is an extremal $W$-code with $\prob[\Ecal]=\frac12$ and it is also easy to see that $\big\|\mathbbm{1}_{\Ecal}-\frac12\big\|_{U_2} \meg \frac12$. Thus, for any nonempty loopless graph~$H$, the parity of the number of edges of $H$ characterizes the Fourier uniformity of extremal $H$-codes.
\end{rem}

The following theorem---which is the second main result of this paper---complements Theorem~\ref{thm:MainEvenEdges} and shows that extremal codes for Problem \ref{p1.8} possess quite strong pseudorandomness properties.

\begin{thm}[Higher order uniformity of extremal $\mathrm{HJ}$-codes] \label{thm:MainCliques}
Let $\Ccal^\circ$ be as in \eqref{e1.6}, namely, $\Ccal^\circ$ is the collection of all cliques with all possible self-loops. Then, for every integer $d\geqslant 2$ and every $\varepsilon>0$, there exists a positive integer $n_0=n_0(d,\varepsilon)$ such that, for all $n\geqslant n_0$, if\, $\Gcal\subseteq \mathbb{F}_2^{\binom{[n]}{\leqslant 2}}$ is an extremal\, $\Ccal^\circ$-$\mathrm{HJ}$-code,~then
\begin{equation} \label{e1.14}
\big\|\mathbbm{1}_{\Gcal}-\prob[\Gcal]\big\|_{U_d} \mik \varepsilon.
\end{equation}
\end{thm}

\begin{rem}[Quantitative estimates] \label{r1.16}
The proof of Theorem \ref{thm:MainEvenEdges} is effective, though in order to obtain explicit estimates for $n_0(\varepsilon,\Hcal)$ one needs to control the rate of convergence of the sequence $\big(\delta_n(\Hcal)\big)$ introduced in \eqref{e1.8}; see Lemma \ref{l3.4} in the main text. This can be done with a standard regularity argument---see, e.g., \cite[Proposition 1.11]{Tao08}---but it will ultimately lead to tower-type bounds. Using the resolution of Marton's conjecture \cite[Corollary 1.6]{GGMT25}, one can also effectivize the case ``$d=3$" of Theorem \ref{thm:MainCliques} but, as expected, the bounds one obtains in this case are also tower-type. Finally, the case ``$d\meg 4$" of Theorem \ref{thm:MainCliques} relies on the inverse theorem for the $U_d\text{-norm}$ due to Tao--Ziegler \cite{TZ12} which remains, at present, ineffective; of course, any progress on the quantitative aspects of the inverse theorem will also yield, fairly straightforwardly, quantitative estimates for Theorem \ref{thm:MainCliques}.
\end{rem}

\section{Subspaces} \label{sec2}

An important ingredient of the proofs of Theorems \ref{thm:MainEvenEdges} and \ref{thm:MainCliques} is the isolation of certain affine subspaces of $\Ftwo{\binom{[n]}{2}}$ and $\Ftwo{\binom{[n]}{\mik 2}}$ that have a particular combinatorial significance. Our goal in this section is to introduce these subspaces.

\subsection{Central embeddings and central subspaces for the model of loopless graphs} \label{subsec2.1}

We start with the model $\mathbb{F}_2^{\binom{[n]}{2}}$ of all loopless graphs on $[n]$. First, we need to introduce some pieces of notation.

Fix a positive integer $n$, and let $I$ be a nonempty subset of $[n]$. Set $m=|I|$, let $f$ denote the unique strictly increasing map from $[m]$ onto~$I$, and define $\mathrm{Id}_I\colon \mathbb{F}_2^{\binom{[m]}{2}} \to \mathbb{F}_2^{\binom{[n]}{2}}$ by setting for every $x\in \mathbb{F}_2^{\binom{[m]}{2}}$,
\begin{enumerate}
\item[$\bullet$] $\mathrm{Id}_I (x)\big(\{i,j\}\big):= x\big(\{f^{-1}(i),f^{-1}(j)\}\big)$ if $\{i,j\} \in \binom{I}{2}$, and
\item[$\bullet$] $\mathrm{Id}_I (x)\big(\{i,j\}\big)=0$ if $\{i,j\} \in \binom{[n]}{2} \setminus \binom{I}{2}$.
\end{enumerate}
Notice that the image of $\mathrm{Id}_I$ is a vector subspace of $\mathbb{F}_2^{\binom{[n]}{2}}$.

Next, given two positive integers $n\geqslant m$, we say that a map $e\colon \mathbb{F}_2^{\binom{[m]}{2}} \to \mathbb{F}_2^{\binom{[n]}{2}}$ is a \emph{central embedding} if it is of the form $e = \mathrm{Id}_I +c$, where $I\in \binom{[n]}{m}$
and $c \in \mathbb{F}_2^{\binom{[n]}{2}}$ with $c(p)=0$ for all $p\in \binom{I}{2}$. We define
\begin{enumerate}
\item[$\bullet$] the \emph{dimension} $\dim(e)$ of $e$ to be the positive integer $m$;
\item[$\bullet$] the \emph{support} $\mathrm{supp}(e)$ of $e$ to be the set $I$;
\item[$\bullet$] the \emph{constant part} $\mathrm{const}(e)$ of $e$ to be the element $c$.
\end{enumerate}
Note that the dimension, the support and the constant part of $e$ uniquely determine $e$.

\begin{defn}[Central subspace] \label{d2.1}
A \emph{central subspace} $V$ of\, $\mathbb{F}_2^{\binom{[n]}{2}}$ is defined to be the image of a central embedding; note that this central embedding is unique and shall be denoted by $e_V$. We also define the \emph{dimension} $\mathrm{dim}(V)$, the \emph{support} $\mathrm{supp}(V)$ and the \emph{constant part} $\mathrm{const}(V)$ of a central subspace $V$ to be the dimension, the support and the constant part of $e_V$, respectively.
\end{defn}

\subsection{$\mathrm{HJ}$-embeddings and $\mathrm{HJ}$-subspaces for the model of graphs with self-loops} \label{subsec2.2}

We proceed to define the appropriate notion of a subspace for the model $\mathbb{F}_2^{\binom{[n]}{\mik 2}}$ of graphs on $[n]$ with self-loops. This class of subspaces is, actually, much richer than the class of central subspaces defined above.

Let $n\geqslant m$ be positive integers, and let $\boldsymbol{I} = (I_i)_{i=1}^m$ be a finite sequence of pairwise disjoint nonempty subsets of $[n]$ with $\min(I_1)< \dots < \min(I_m)$. For every $q\in \binom{[m]}{\leqslant 2}$, let
$b_q\in \mathbb{F}_2^{\binom{[n]}{\leqslant 2}}$ denote the indicator of the set
\begin{equation} \label{e2.1}
\bigg\{ p\in \binom{[n]}{\leqslant 2} \colon p\subseteq \bigcup_{i\in q}I_i \text{ and } p\cap I_i\neq\emptyset \text{ for all } i\in q\bigg\};
\end{equation}
we define $\mathrm{Id}_{\boldsymbol{I}}\colon \mathbb{F}_2^{\binom{[m]}{\leqslant 2}} \to \mathbb{F}_2^{\binom{[n]}{\leqslant 2}}$ by
\begin{equation} \label{e2.2}
\mathrm{Id}_{\boldsymbol{I}}(x) := \sum_{q\in \binom{[m]}{\leqslant 2}} x(q)\, b_q.
\end{equation}
We say that a map $e\colon \mathbb{F}_2^{\binom{[m]}{\leqslant 2}} \to \mathbb{F}_2^{\binom{[n]}{\leqslant 2}}$ is a \emph{HJ-embedding}\footnote{As in Definition \ref{d1.1}, the acronym ``HJ" comes from ``Hales--Jewett".} if it is of the form $e = \mathrm{Id}_{\boldsymbol{I}}+c$, where $\boldsymbol{I}=(I_i)_{i=1}^m$ is a finite sequence of pairwise disjoint nonempty subsets of $[n]$ that satisfy $\min(I_1)< \dots < \min(I_m)$ and $c\in \mathbb{F}_2^{\binom{[n]}{\leqslant 2}}$ is such that $c(p) = 0$ for all $p\in \binom{I_1\cup\dots\cup I_m}{\leqslant 2}$. We define
\begin{enumerate}
\item[$\bullet$] the \emph{dimension} $\dim(e)$ of $e$ to be the positive integer $m$;
\item[$\bullet$] the \emph{wildcard sets} $\mathrm{wild}(e)$ of $e$ to be the finite sequence $\boldsymbol{I}$;
\item[$\bullet$] the \emph{constant part} $\mathrm{const}(e)$ of $e$ to be the element $c$.
\item[$\bullet$] the \emph{variable sets} $\mathrm{var}(e) := \big\langle \mathrm{var}(e)_q\colon q\in \binom{[\dim(e)]}{\mik 2}\big\rangle$ by setting, for every $q\in \binom{[\dim(e)]}{\mik 2}$,
\begin{equation} \label{e2.3}
\mathrm{var}(e)_q := \bigg\{ p\in \binom{[n]}{\leqslant 2}\colon p\subseteq \bigcup_{i\in q}I_i \text{ and } p\cap I_i\neq\emptyset \text{ for all } i\in q\bigg\};
\end{equation}
\end{enumerate}
Again, observe that the dimension, the wildcard sets, and the constant part of $e$ uniquely determine $e$.

\begin{defn}[$\mathrm{HJ}$-subspace] \label{d2.2}
A \emph{$\mathrm{HJ}$-subspace} $V$ of $\mathbb{F}_2^{\binom{[n]}{\leqslant 2}}$ is defined to be the image of a $\mathrm{HJ}\text{-embedding}$; as in the case of central subspaces, we note that this $\mathrm{HJ}$-embedding is unique and shall be denoted by~$e_V$. We define the \emph{dimension} $\mathrm{dim}(V)$, the \emph{wildcard sets} $\mathrm{wild}(V)$, the \emph{variable sets} $\mathrm{var}(V)$ and the \emph{constant part} $\mathrm{const}(V)$ of $V$ to be the dimension, the wildcard sets, the variable sets and the constant part of $e_V$, respectively.
\end{defn}

We will also encounter HJ-subspaces that have the following special form.

\begin{defn}[Block $\mathrm{HJ}$-subspace] \label{d2.3}
We say that a $\mathrm{HJ}$-subspace $V$ of $\mathbb{F}_2^{\binom{[n]}{\leqslant 2}}$ is \emph{block} if its wildcard sets $\mathrm{wild}(V)=(I_i)_{i=1}^{\dim(V)}$ are successive, that is, if $\max(I_i)<\min(I_{i+1})$ for every positive integer $i\leqslant \dim(V)-1$.
\end{defn}

\section{Fourier uniformity of extremal graph-codes: proof of Theorem \ref*{thm:MainEvenEdges}} \label{sec3}

\subsection{Discrete Fourier analysis} \label{subsec3.1}

We begin by reviewing some basic facts from discrete Fourier analysis that are needed for the proof of Theorem \ref{thm:MainEvenEdges}; see \cite{HHL19,O'Don14,TV06} for a detailed treatment. In what follows, let $\Ical$ denote a nonempty finite set.

\subsubsection{Norms} \label{subsubsec3.1.1}

Given $1 \leqslant p < \infty$, we associate to every function $f\colon \mathbb{F}_2^{\Ical} \to \mathbb{C}$ the norms
\begin{equation} \label{e3.1}
\|f\|_{L_p}:= \Big(\mathbb{E}\big[|f|^p\big]\Big)^{\frac{1}{p}} \ \ \ \text{ and } \ \ \
\|f\|_{\ell_p}:= \bigg( \sum_{\xi \in \mathbb{F}_2^\Ical} |f(\xi)|^p\bigg)^{\frac{1}{p}}.
\end{equation}
We also define the $L_\infty$-norm and $\ell_\infty$-norm by $\|f\|_{\ell_\infty} = \|f\|_{L_\infty} := \max\big\{ |f(x)|\colon x\in\mathbb{F}_2^{\Ical}\big\}$.

The scalar products associated with the $L_2$-norm and the $\ell_2$-norm will be denoted by $\langle\cdot, \cdot\rangle_{L_2}$ and $\langle\cdot, \cdot\rangle_{\ell_2}$, respectively. We will often denote the vector space of all complex-valued functions on $\Ftwo{\Ical}$ by $L_p(\Ftwo{\Ical})$ or $\ell_p(\Ftwo{\Ical})$ to signal the corresponding norm.

\subsubsection{Characters} \label{subsubsec3.1.2}

Recall that linear functionals $\chi\colon\Ftwo{\Ical} \to \Ftwo{}$ are exactly the maps of the form $x \mapsto |x \cap \xi| \mod 2$, where $\xi \in \Ftwo{\Ical}$. Thus, the characters of the abelian group $\Ftwo{\Ical}$---that is, group homomorphisms $\Ftwo{\Ical} \to \mathbb{C}^\times$---are exactly the \textit{Walsh functions} $\langle w_\xi\colon \xi \in \Ftwo{\Ical}\rangle$ defined by setting, for every $\xi \in \Ftwo{\Ical}$ and every $x \in \Ftwo{\Ical}$,
\begin{equation} \label{e3.2}
w_\xi(x) := (-1)^{|x \cap \xi|};
\end{equation}
the function $w_\xi$ is often referred to as \textit{Walsh function at $\xi$}. Notice that
\begin{enumerate}
\item[(P1)] $w_0 = \mathbbm{1}$,
\item[(P2)] \label{it:ExpectWalsh} $\mathbb{E}[w_\xi]= 0$ for all $\xi \neq 0$, and
\item[(P3)] \label{it:ProductWalsh} $w_\xi(x) w_\zeta(x) = w_{\xi+\zeta}(x)$ for all $\xi,\zeta,x\in\Ftwo{\Ical}$.
\end{enumerate}
In particular, the family of Walsh functions is an orthonormal basis of $L_2(\Ftwo{\Ical})$.

\begin{defn}[Fourier transform] \label{d3.1}
For any function $f \colon \Ftwo{\Ical}\to \mathbb{C}$, its \textit{Fourier transform} $\widehat f \colon \Ftwo{\Ical} \to \mathbb{C}$ is defined by setting $\widehat f(\xi) := \langle f, w_\xi\rangle_{L_2}$.
\end{defn}

Observe that $\widehat{\mathbbm{1}_{\Ftwo{\Ical}}}$ is the indicator of $\{0\}$, and $\widehat f(0) = \mathbb{E}[f]$ for any $f \colon \Ftwo{\Ical} \to \mathbb{C}$. Moreover, since the Walsh system is an orthonormal basis of $L_2(\Ftwo{\Ical})$, we have, for any
$f \colon \Ftwo{\Ical} \to \mathbb{C}$,
\begin{enumerate}
\item[(P4)] \label{Inversion} (Fourier inversion formula) $f(x) = \sum_{\xi \in \mathbb{F}_2^\Ical} \widehat f(\xi) w_\xi(x)$ for all $x \in \mathbb{F}_2^\Ical$, and
\item[(P5)] \label{Parseval} (Parseval's identity) $\|\widehat f\|_{\ell_2} = \|f\|_{L_2}$.
\end{enumerate}
We also note that we shall consider complex functions on $\Ftwo{\Ical}$ that have a physical or combinatorial significance---such as, indicators of subsets of $\mathbb{F}_2^{\Ical}$---as elements of the $L_p$-spaces, while their Fourier transforms are viewed as elements of the $\ell_p$-spaces; this is done in order to avoid normalization constants.

\subsubsection{Spectral interpretation of the $U_2$-norm} \label{subsubsec3.1.3}

The following basic fact gives a spectral interpretation of the second Gowers uniformity norm; see, e.g., \cite[(11.3)]{TV06}.

\begin{fact} \label{prop:U2ByFourier}
For any $f \colon \Ftwo{\Ical} \to \mathbb{C}$, we have $\|f\|_{U_2} = \|\widehat f\|_{\ell_4}$.
\end{fact}

We will need the following consequence of Fact \ref{prop:U2ByFourier}.

\begin{cor} \label{cor:U2FromEllInf}
For any $f \colon \Ftwo{\Ical} \to \mathbb{C}$ with $\|f\|_{L_2} \leqslant 1$, we have
\begin{equation} \label{e3.3}
\|\widehat f\|_{\ell_\infty} \leqslant \|f\|_{U_2} \leqslant \sqrt{\|\widehat f\|_{\ell_\infty}}.
\end{equation}
\end{cor}

\begin{proof}
By Fact \ref{prop:U2ByFourier}, it is enough to prove that $\|\widehat f\|_{\ell_4}^2 \leqslant \|\widehat f\|_{\ell_\infty}$. Indeed, by our assumptions,
\begin{equation} \label{e3.4}
\|\widehat f\|_{\ell_4}^4 = \sum_{\xi \in \Ftwo{\Ical}} |\widehat f(\xi)|^4 \leqslant \|\widehat f\|_{\ell_\infty}^2 \sum_{\xi \in \Ftwo{\Ical}} |\widehat f(\xi)|^2 = \|\widehat f\|_{\ell_\infty}^2\|\widehat f\|_{\ell_2}^2
\stackrel{(\hyperref[Parseval]{\mathrm{P5}})}{=}
\|\widehat f\|_{\ell_\infty}^2\|f\|_{L_2}^2 \mik \|\widehat f\|_{\ell_\infty}^2. \qedhere
\end{equation}
\end{proof}

\subsection{Proof of Theorem \ref{thm:MainEvenEdges}}

The following lemma is the main step of the proof.

\begin{lem} \label{l3.4}
Let $\Hcal$ be a collection of nonempty loopless graphs, each with an an even number of edges. Also let $n,m\geqslant 2$ be integers with $n \geqslant 4^m$, and let\, $\mathcal{G} \subseteq \mathbb{F}_2^{\binom{[n]}{2}}$ be an extremal $\mathcal{H}$-code, that is, $\prob[\mathcal{G}]= \delta_n(\Hcal)$, where $\delta_n(\Hcal)$ is as in \eqref{e1.8}. Then,
\begin{equation} \label{e3.5}
\big\| \widehat{(\mathbbm{1}_{\Gcal}-\prob[\Gcal])}\big\|_{\ell_\infty} \leqslant \delta_m(\Hcal)-\delta_n(\Hcal).
\end{equation}
\end{lem}

Granting Lemma \ref{l3.4}, let us complete the proof of Theorem \ref{thm:MainEvenEdges}.

\begin{proof}[Proof of Theorem \ref{thm:MainEvenEdges} assuming Lemma \ref{l3.4}]
Fix $\varepsilon >0$. By Fact \ref{f1.11}, there exists an integer $n_1=n_1(\varepsilon,\Hcal)$ such that $\delta_k(\Hcal)-\delta_\ell(\Hcal) \mik\varepsilon^2$ for every pair of integers $\ell \meg k\meg n_1$. Set $n_0:=4^{n_1}$; let $n\meg n_0$, and let $\Gcal\subseteq \Ftwo{\binom{[n]}{2}}$ be an extremal $\Hcal$-code. Then, by Lemma \ref{l3.4} applied for ``$m=n_1$", we obtain~that
\begin{equation} \label{e3.6}
\big\| \mathbbm{1}_{\Gcal}-\prob[\Gcal]\big\|_{U_2} \stackrel{\eqref{e3.3}}{\mik} \big\| \widehat{(\mathbbm{1}_{\Gcal}-\prob[\Gcal])}\big\|_{\ell_\infty}^{1/2} \stackrel{\eqref{e3.5}}{\mik} \sqrt{\delta_{n_1}(\Hcal)-\delta_n(\Hcal)}\mik \varepsilon . \qedhere
\end{equation}
\end{proof}

We proceed to the proof of Lemma \ref{l3.4}.

\begin{proof}[Proof of Lemma \ref{l3.4}]
For notational convenience, set $f:=\mathbbm{1}_{\Gcal}-\prob[\Gcal]$; notice that $\widehat{\mathbbm{1}_{\Gcal}}(0) = \prob[\Gcal]$, $\widehat{f}(0) = 0$, and $\widehat{f}(G)= \widehat{\mathbbm{1}_{\Gcal}}(G)$ for all $G \in \mathbb{F}_2^{\binom{[n]}{2}} \setminus \{0\}$. Let $G_0\in\mathbb{F}_2^{\binom{[n]}{2}}$ such that $|\widehat{f}(G_0)| = \|\widehat{f}\|_{\ell_\infty}$ and define, for every $i\in\{0,1\}$,
\begin{equation} \label{e3.7}
\Ecal_i := \bigg\{ G\in\mathbb{F}_2^{\binom{[n]}{2}} \colon |G_0\cap G| = i \mod 2\bigg\}.
\end{equation}
Since $G_0$ is nonempty, $\prob[\mathcal{E}_0] = \prob[\mathcal{E}_1] = \frac{1}{2}$. It follows that
\begin{equation} \label{e3.8}
\prob[\Gcal] = \frac{1}{2} \prob\big[ \mathcal{G}\, \big|\, \mathcal{E}_0\big] + \frac{1}{2} \prob\big[ \mathcal{G}\, \big|\,\mathcal{E}_1\big] \ \ \ \text{ and } \ \ \
\widehat{f}(G_0) = \widehat{\mathbbm{1}_{\Gcal}}(G_0) = \frac{1}{2} \prob\big[\mathcal{G}\, \big|\, \mathcal{E}_0\big] - \frac{1}{2}
\prob\big[ \mathcal{G}\, \big|\, \mathcal{E}_1\big].
\end{equation}
Therefore, there exists $i_0\in\{0,1\}$ such that
\begin{equation} \label{eq2.01}
\prob\big[\mathcal{G}\, \big|\, \mathcal{E}_{i_0}\big] =
\prob[\mathcal{G}] + \|\widehat{f}\|_{\ell_\infty} = \delta_n(\Hcal) + \|\widehat{f}\|_{\ell_\infty}.
\end{equation}
By the classical Erd\H{o}s--Szekeres bound \cite{ES35} and our assumption that $4^m\leqslant n$, we may select a subset $A$ of $[n]$ of cardinality $m$ such that either $\binom{A}{2}\cap G_0 = \emptyset$ or $\binom{A}{2}\subseteq G_0$. We consider cases.

\subsection*{Case 1: $\binom{A}{2}\cap G_0 = \emptyset$} Set
\begin{equation} \label{e3.10}
\mathcal{X} := \bigg\{ x\in \mathbb{F}_2^{\binom{[n]}{2} \setminus \binom{A}{2}} \colon |x\cap G_0| = i_0 \mod 2\bigg\},
\end{equation}
and for every $x\in\mathcal{X}$, define
\begin{equation} \label{e3.11}
V_x := \bigg\{x\cup y \colon y\subseteq \binom{A}{2} \bigg\}.
\end{equation}
Notice that the family $\langle V_x\colon x\in\mathcal{X}\rangle$ is a partition of $\mathcal{E}_{i_0}$ into $m$-dimensional central subspaces. Thus,
\begin{equation} \label{eq2.02}
\prob\big[ \mathcal{G}\, \big|\, \mathcal{E}_{i_0}\big] = \underset{x\in\Xcal}{\ave}
\Big[ \prob\big[ \mathcal{G}\, \big|\, V_x\big]\Big].
\end{equation}
Finally, let $x\in\Xcal$ be arbitrary and, as in Definition \ref{d2.1}, let $e_{V_x}\colon \Ftwo{\binom{[m]}{2}}\to V_x$ denote the central embedding associated with $V_x$; then, observe that the family $e_{V_x}^{-1}\big(\Gcal\big)\subseteq \Ftwo{\binom{[m]}{2}}$ is an $\Hcal$-code which in turn implies---since $e_{V_x}$ is a bijection---that $\prob\big[\Gcal\,\big|\, V_x]\mik \delta_m(\Hcal)$. Hence, by \eqref{eq2.01}, \eqref{eq2.02} and the previous observations, there exists $x_0\in \mathcal{X}$ such that
\begin{equation} \label{e3.13}
\delta_m(\Hcal) \geqslant \prob\big[ \Gcal\,\big|\, V_{x_0}\big] \geqslant \prob\big[ \Gcal\,\big|\, \mathcal{E}_{i_0}\big]
= \delta_n(\Hcal)+ \|\widehat{f}\|_{\ell_\infty},
\end{equation}
as desired.

\subsection*{Case 2: $\binom{A}{2}\subseteq G_0$} The argument is similar, but additionally uses the parity assumption on~$\Hcal$. Set $\mathcal{X}:= \mathbb{F}_2^{\binom{[n]}{2} \setminus \binom{A}{2}}$ and, for every $x\in\mathcal{X}$, let $V_x$ be the $m$-dimensional central subspace of $\Ftwo{\binom{[n]}{2}}$ defined in \eqref{e3.11}. Since the family $\langle V_x\colon x\in\mathcal{X}\rangle$ is a partition of $\mathbb{F}_2^{\binom{[n]}{2}}$ into sets of equal size, we can find $x_0 \in \mathcal{X}$ such that
\begin{equation} \label{e3.14}
\prob\big[ \Gcal \cap \mathcal{E}_{i_0}\,\big|\, V_{x_0}\big] \geqslant \prob[ \mathcal{G} \cap \mathcal{E}_{i_0}\big] = \frac{1}{2} \prob\big[ \Gcal\,\big|\, \mathcal{E}_{i_0}\big]
\stackrel{\eqref{eq2.01}}{=} \frac{1}{2} \big(\delta_n(\Hcal)+ \|\widehat{f}\|_{\ell_\infty}\big).
\end{equation}
Select a subset $e$ of $\binom{A}{2}$ of odd cardinality, and define
\begin{equation} \label{e3.15}
\mathcal{G}_0 := \Gcal \cap \mathcal{E}_{i_0} \cap V_{x_0}  \ \ \ \text{ and } \ \ \ \mathcal{G}_{1} := \big\{z+e \colon z\in \mathcal{G}_{0}\big\}.
\end{equation}
Finally, set $\mathcal{G}' := \mathcal{G}_0 \cup \mathcal{G}_1$. Notice that $\mathcal{G}' \subseteq V_{x_0}$ and
\begin{equation} \label{e3.16}
\prob\big[ \Gcal_1\,\big|\, V_{x_0}\big] = \prob\big[ \Gcal_0\,\big|\, V_{x_0}\big] \stackrel{\eqref{e3.14}}{\geqslant} \frac{1}{2} \big(\delta_n(\Hcal)+ \|\widehat{f}\|_{\ell_\infty}\big).
\end{equation}
Observe that, since $\binom{A}{2} \subseteq G_0$, we have $G \setminus G_0 = x_0 \setminus G_0$ for all $G \in \Gcal_0$. In particular, for such a $G$, we have
\begin{equation} \label{e3.17}
|G| = |G \cap G_0| + |x_0 \setminus G_0| = i_0 + |x_0 \setminus G_0| \mod 2.
\end{equation}
Letting $j_0 = i_0 + |x_0 \setminus G_0| \mod 2$, it follows that, for $i = 0, \, 1$ and $G \in \Gcal_i$,
\begin{equation} \label{e3.18}
|G| = j_0 + i \mod 2.
\end{equation}
In particular, $\Gcal_0 \cap \Gcal_1 = \varnothing$.

\begin{claim} \label{c3.5}
The family $\Gcal'$ is an $\Hcal$-code.
\end{claim}

\begin{proof}[Proof of Claim \ref{c3.5}]
Notice, first, that $\mathcal{G}_0$ and $\mathcal{G}_1$ are both $\mathcal{H}$-codes. Next, let $x\in\mathcal{G}_0$ and $y\in\mathcal{G}_1$; since $|x|$ and $|y|$ have different parities, the cardinality of $x+y$ must be odd which further implies, after recalling that $\mathcal{H}$ consists of nonempty graphs with an even number of edges, that $x+y$ is not isomorphic to a graph in $\Hcal$. Thus, $\Gcal'$ is an $\Hcal$-code.
\end{proof}

Since $\mathcal{G}'\subseteq V_{x_0}$ and $V_{x_0}$ is a central subspace of dimension $m$, by Claim \ref{c3.5} and arguing as in Case 1, we conclude that
\begin{equation} \label{e3.20}
\delta_m(\Hcal) \geqslant \prob\big[ \mathcal{G}'\, \big| \, V_{x_0}\big]
= \prob\big[ \Gcal_0\,\big|\, V_{x_0}\big] + \prob\big[ \Gcal_1\,\big|\, V_{x_0}\big] \stackrel{\eqref{e3.16}}{\geqslant} \delta_n(\Hcal)+ \|\widehat{f}\|_{\ell_\infty}.
\end{equation}
The above cases are exhaustive, and so the entire proof is completed.
\end{proof}

\section{Nonclassical polynomials, and the inverse theorem in characteristic two} \label{sec4}

Throughout this section, let $\Ical$ be a nonempty finite set; also let $\mathbb{T} :=\mathbb{R}/\mathbb{Z}$ denote the torus. For every function $f \colon \Ftwo{\Ical} \to \mathbb{T}$ and every $h\in \Ftwo{\Ical}$, let $\Delta_h f\colon \Ftwo{\Ical}\to \mathbb{T}$ denote the \emph{(discrete) derivative of $f$ in the direction $h$}, that is,
\begin{equation} \label{e4.1}
\Delta_hf(x):= f(x+h)-f(x),
\end{equation}
for all $x\in \Ftwo{\Ical}$.

\begin{defn}[Nonclassical polynomials] \label{d4.1}
Let $d$ be a nonnegative integer. We say that a map $P\colon  \Ftwo{\Ical}\to \mathbb{T}$ is a \emph{nonclassical polynomial of degree at most $d$} if for every $h_1,\dots,h_{d+1}, x\in \Ftwo{\Ical}$,
\begin{equation} \label{e4.2}
\Delta_{h_1}\cdots \Delta_{h_{d+1}} P(x)=0.
\end{equation}
The \emph{degree $\deg(P)$ of $P$} is defined to be the least nonnegative integer $d$ such that \eqref{e4.2} holds~true.
\end{defn}

We will need the following characterization of nonclassical polynomials due to Tao--Ziegler \cite[Lemma 1.6]{TZ12}.

\begin{lem} \label{lem:nonClassicalPolynomials}
Let $d$ be a nonnegative integer, and let $P\colon \Ftwo{\Ical}\to\mathbb{T}$ be a function. Then $P$ is a nonclassical polynomial of degree at most $d$ if and only if there exist $\alpha\in\mathbb{T}$ and, for every\footnote{Recall that by $\binom{\Ical}{\mik d}$ we denote the set of all \emph{nonempty} subsets of $\Ical$ with cardinality at most $d$.} $S\in\binom{\Ical}{\leqslant d}$, a coefficient $\lambda_S \in \big\{\frac{j}{2^{d-|S|+1}}\colon j=0,\dots,2^{d-|S|+1}-1\big\}$ such that\footnote{In \eqref{eq:4.1.07}, each $x(i) \in \mathbb{F}_2$ is identified with an element of\, $\mathbb{R}$ (either $0$ or $1$, in the obvious way).}, for all $x\in\Ftwo{\Ical}$,
\begin{equation}\label{eq:4.1.07}
P(x) = \alpha + \sum_{S\in \binom{\Ical}{\leqslant d}} \lambda_S \prod_{i\in S}x(i) \mod 1.
\end{equation}
\end{lem}

An elementary argument---see, e.g., \cite[Theorem 6.7]{HHL19}---shows that for every function $f\colon \Ftwo{\Ical}\to \mathbb{C}$ and every nonclassical polynomial $P\colon \Ftwo{\Ical}\to\mathbb{T}$ of degree at most $d$,
\begin{equation} \label{4.v2.new}
\Big| \underset{x\in\Ftwo{\Ical}}{\ave} \Big[f(x) \exp\big(2\pi i P(x)\big)\Big] \Big| \leqslant \|f\|_{U_{d+1}}.
\end{equation}
The following deep result due to Tao--Ziegler \cite[Theorem 1.10]{TZ12} provides, essentially, a converse of this estimate that characterizes functions $f\colon \Ftwo{\Ical}\to\mathbb{C}$ of modulus at most one with large $U_d$-norm.

\begin{thm}[Inverse theorem in characteristic two] \label{thm:inverse}
For every integer $d\meg 1$ and every $\varepsilon>0$, there exists a constant $\delta = \delta(d,\varepsilon)>0$ with the following property. Let $\overline{\mathbb{D}} := \{z\in \mathbb{C}\colon |z|\leqslant 1\}$ denote the closed unit disc, and let $f\colon
\Ftwo{\Ical}\to \overline{\mathbb{D}}$ be a function with $\|f\|_{U_{d+1}}\geqslant\varepsilon$. Then there exists a nonclassical
polynomial\, $P\colon \Ftwo{\Ical}\to\mathbb{T}$ of degree at most $d$ such~that
\begin{equation} \label{e4.4}
\Big| \underset{x\in\Ftwo{\Ical}}{\ave}\Big[f(x) \exp\big(2\pi i P(x)\big)\Big] \Big| \geqslant\delta.
\end{equation}
\end{thm}

\section{Partitioning nonclassical polynomials} \label{sec5}

The following theorem is one of the main ingredients of the proof of Theorem \ref{thm:MainCliques}. It shows that, for any nonclassical polynomial $P\colon \mathbb{F}_2^{\binom{[n]}{\leqslant 2}}\to\mathbb{T}$, one can almost entirely partition $\mathbb{F}_2^{\binom{[n]}{\leqslant 2}}$ into $\mathrm{HJ}$-subspaces of large dimension such that $P$ is constant on each of them.

\begin{thm}[Partitioning nonclassical polynomials] \label{thm_part_nclas_pol}
Let $\eta>0$, and let $m,d$ be positive integers. Then there exists a positive integer $n_2 = n_2(\eta, m,d)$ with the following property.
Let $n\geqslant n_2$ be an integer, and let $P\colon \Ftwo{\binom{[n]}{\mik 2}}\to \mathbb{T}$ be a nonclassical polynomial of degree at most $d$. Then there exists a collection $\mathcal{V}$ of pairwise disjoint $m$-dimensional block\, $\mathrm{HJ}$-subspaces of\, $\mathbb{F}_2^{\binom{[n]}{\leqslant2}}$ such that
\begin{enumerate}
\item[(i)] $\prob\Big[\mathbb{F}_2^{\binom{[n]}{\mik 2}}\setminus (\cup\mathcal{V})\Big] \mik \eta$, and
\item[(ii)] for every $V\in\mathcal{V}$, the polynomial $P$ is constant on $V$.
\end{enumerate}
\end{thm}

\begin{rem} \label{r5.2}
The density polynomial Hales--Jewett conjecture predicts that, for \emph{any} function $f\colon \Ftwo{\binom{[n]}{\mik 2}}\to\mathbb{C}$ with $\|f\|_{L_2}\mik 1$, one can almost entirely partition $\mathbb{F}_2^{\binom{[n]}{\leqslant 2}}$ into $\mathrm{HJ}$-subspaces of large dimension such that $f$ is almost constant on each of them. While Theorem \ref{thm_part_nclas_pol} falls short of proving such a powerful statement, it can be viewed as evidence towards the conjecture.
\end{rem}

\subsection{Reduction to integer polynomials} \label{subsec5.1}

For the proof of Theorem \ref{thm_part_nclas_pol}, it is convenient to work with the following larger class of polynomials.

\begin{defn}[Integer polynomials] \label{d5.3}
Let $d$ be a nonnegative integer, and let $k,n$ be positive integers. We say that a function\footnote{Here, $\mathbb{Z}_{2^k}:=\mathbb{Z}/2^k\mathbb{Z}$ denotes the cyclic group of integers mod $2^k$.} $Q\colon \mathbb{F}_2^{\binom{[n]}{\leqslant 2}} \to \mathbb{Z}_{2^k}$ is a \emph{$k$-integer polynomial of degree at most~$d$} if it is of the form\footnote{Similarly, in \eqref{eq3.01}, each $x(e) \in \mathbb{F}_2$ is identified with an element of\, $\mathbb{Z}_{2^k}$ (either $0$ or $1$, in the obvious way).}
\begin{equation} \label{eq3.01}
Q(x) = \alpha + \sum_{F\in \binom{\binom{[n]}{\leqslant 2}}{\leqslant d}}\lambda_F \prod_{e\in F}x(e) \mod 2^k,
\end{equation}
where $\alpha\in \mathbb{Z}_{2^k}$ and $\lambda_F\in\mathbb{Z}_{2^k}$ for every $F\in \binom{\binom{[n]}{\leqslant2}}{\leqslant d}$. We say that $Q$ is of \emph{degree $d$}, and we write $\deg(Q) = d$, if $d$ is the smallest nonnegative integer such that\, $Q$ admits a representation of the form \eqref{eq3.01}.
\end{defn}

The following lemma makes the link between nonclassical and integer polynomials.

\begin{lem} \label{l5.4}
Let $n, d$ be positive integers. Then every nonclassical polynomial $P\colon \Ftwo{\binom{[n]}{\mik 2}}\to \mathbb{T}$ of degree at most $d$ can be written as  $f \circ Q$, where $Q \colon \Ftwo{\binom{[n]}{\leqslant2}}\to \mathbb{Z}_{2^d}$ is a $d$-integer polynomial of degree at most $d$ and $f \colon \mathbb{Z}_{2^d} \to \mathbb{T}$ is a map.
\end{lem}

\begin{proof}
For notational convenience, set $\Ical:= \binom{[n]}{\mik 2}$. By Lemma \ref{lem:nonClassicalPolynomials}, there exist $\alpha\in\mathbb{T}$ and, for every $S\in\binom{\Ical}{\leqslant d}$, a coefficient $\lambda_S \in \big\{\frac{j}{2^{d-|S|+1}}\colon j=0,\dots,2^{d-|S|+1}-1\big\}$ such that $P$ is represented~as
\begin{equation} \label{e5.2}
P(x) = \alpha + \sum_{S\in \binom{\Ical}{\leqslant d}} \lambda_S \prod_{e\in S} x(e) \mod 1.
\end{equation}
For every $S \in \binom{\Ical}{\mik d}$, set $\lambda'_S := 2^d \lambda_S$. Identifying $\mathbb{Z}_{2^d}$ with $\{0,\dots,2^d-1\}$ (in the obvious way), we see that $\lambda'_S\in \mathbb{Z}_{2^d}$ for every $S\in\binom{\Ical}{\mik d}$. Define $Q \colon \Ftwo{\Ical}\to \mathbb{Z}_{2^d}$ by setting
\begin{equation} \label{e5.3}
Q(x) := \sum_{S\in \binom{\Ical}{\leqslant d}} \lambda'_S \prod_{e\in S} x(e) \mod 2^d.
\end{equation}
Clearly, $Q$ is a $d$-integer polynomial of degree at most $d$ and $P(x) = \alpha +2^{-d} Q(x) \mod 1$ for every $x\in \mathbb{F}_2^{\Ical}$.
\end{proof}

Lemma \ref{l5.4} immediately shows that Theorem \ref{thm_part_nclas_pol} follows from the following corresponding partitioning theorem for integer polynomials.

\begin{thm}[Partitioning integer polynomials] \label{thm_part_int_pol}
Let $\eta>0$, and let $m,d, k$ be positive integers. Then there exists a positive integer $n_3 = n_3(\eta, m,d, k)$ with the following property. Let $n\meg n_3$ be an integer, and let $Q\colon \Ftwo{\binom{[n]}{\mik 2}}\to \mathbb{Z}_{2^k}$ be a $k$-integer polynomial of degree at most $d$. Then there exists a collection $\mathcal{V}$ of pairwise disjoint $m$-dimensional block\, $\mathrm{HJ}$-subspaces of\, $\mathbb{F}_2^{\binom{[n]}{\leqslant2}}$ such that
\begin{enumerate}
\item[(i)] $\prob\Big[ \mathbb{F}_2^{\binom{[n]}{\mik 2}} \setminus (\cup\mathcal{V})\Big] \mik \eta$, and
\item[(ii)] for every $V\in\mathcal{V}$, the polynomial $Q$ is constant on $V$.
\end{enumerate}
\end{thm}

\subsection{Proof of Theorem \ref{thm_part_int_pol}} \label{subsec5.2}

We start with the following lemma that shows that the class of integer polynomials is stable under restrictions to $\mathrm{HJ}$-subspaces.

\begin{lem} \label{l5.5}
Let $n,k$ be positive integers, and let $Q\colon \mathbb{F}_2^{\binom{[n]}{\leqslant2}} \to \mathbb{Z}_{2^k}$ be a $k$-integer polynomial. Also let $V$ be a $\mathrm{HJ}$-subspace of\, $\mathbb{F}_2^{\binom{[n]}{\leqslant2}}$, set $m:=\dim(V)$, and let $e_V\colon \mathbb{F}_2^{\binom{[m]}{\mik 2}}\to V$ be its associated embedding. (See Definition \ref{d2.2}.) Then the function $Q \circ e_V\colon \mathbb{F}_2^{\binom{[m]}{\mik 2}} \to \mathbb{Z}_{2^k}$ is also a $k$-integer polynomial with $\deg(Q \circ e_V) \mik \deg(Q)$.
\end{lem}

\begin{proof}
Write $Q$ as in \eqref{eq3.01} with $d:=\mathrm{deg}(Q)$. Set $\mathcal{I}_n := \binom{[n]}{\leqslant 2}$, $\mathcal{I}_m := \binom{[m]}{\leqslant 2}$ and\footnote{Here, we identify the vector $\mathrm{const}(V) \in \mathbb{F}_2^{\mathcal{I}_n}$ with a subset of $\mathcal{I}_n$.}
\begin{equation} \label{e5.4}
\alpha' :=\alpha + \sum_{F\in\binom{\mathrm{const}(V)}{\leqslant d}} \lambda_F \mod 2^k.
\end{equation}
Moreover, for any $T\in\binom{\mathcal{I}_m}{\leqslant d}$, set
\begin{equation}
\label{e5.5} \lambda'_T := \sum_{F\in \Hcal_T} \lambda_F \mod 2^k,
\end{equation}
where $\Hcal_T$ is the ``hitting" set defined by
\begin{align} \label{e5.6}
\Hcal_T:=\bigg\{ F\in\binom{\mathcal{I}_n}{\leqslant d} \colon & F\not\subseteq\mathrm{const}(V),
F\subseteq\mathrm{const}(V)\cup\bigcup_{q\in \mathcal{I}_m}\mathrm{var}(V)_q, \\
& \ \ \ \ \ \ \ \ \ \ \ \ \ \ \text{and } T=\{q\in \mathcal{I}_m\colon \mathrm{var}(V)_q\cap F\neq\emptyset\}\bigg\}. \nonumber
\end{align}
Then observe that with these choices we have, for any $y\in\mathbb{F}_2^{\mathcal{I}_m}$,
\begin{equation} \label{e5.7}
Q\circ e_V (y) = \alpha' + \sum_{T\in\binom{\mathcal{I}_m}{\mik d}} \lambda'_T \prod_{q\in T} y(q) \mod 2^k;
\end{equation}
therefore, $Q\circ e_V$ is a $k$-integer polynomial with $\deg(Q \circ e_V) \mik \deg(Q)$.
\end{proof}

An important ingredient of the proof of Theorem \ref{thm_part_int_pol} is Ramsey's classical theorem \cite{Ra30}.

\begin{thm}[Ramsey theorem] \label{thm-Ramsey-main}
For every triple $\ell,m,k$ of positive integers with $m \meg \ell\meg 2$, there exists a positive integer $n_0=n_0(\ell,m,k)$ with the following property. If $n\geqslant n_0$ is an integer and $c\colon \binom{[n]}{\ell}\to [k]$, then there exists $X\in\binom{[n]}{m}$ such
that $c$ is constant on $\binom{X}{\ell}$. The least integer $n_0$ with this property is denoted by $\mathrm{R}(\ell,m,k)$.
\end{thm}

Ramsey theorem, Theorem \ref{thm-Ramsey-main}, will be used to ``canonize" the coefficients of a given integer polynomial, in the sense of the following definition.

\begin{defn}[Types and canonical collections] \label{d5.7}
Let $n\meg 2$ be an integer.
\begin{enumerate}
\item[(i)] For any nonempty set $F$ of nonempty subsets of $[n]$, we define the \emph{type $\tau(F)$ of $F$} as follows. Let\, $\cup F$ denote the union of all members of $F$, set $\ell:=|\cup F|$, write $\cup F$ in increasing order as $\{u_1<\cdots<u_\ell\}$, and set
\begin{equation} \label{e5.8}
\tau(F) := \big\{ S\subseteq [\ell] \colon \{u_i:i\in S\}\in F\big\}.
\end{equation}
\item[(ii)] Let $d,k$ be positive integers, and let $\boldsymbol{\lambda}=\big\langle \lambda_F\colon F\in \binom{\binom{[n]}{\mik 2}}{\mik d} \big\rangle$ be a collection of elements of~$\mathbb{Z}_{2^k}$. Also let $X$ be a nonempty subset of\, $[n]$. We say that $\boldsymbol{\lambda}$ is \emph{canonical in $X$} if $\lambda_{F_1}=\lambda_{F_2}$ for any pair $F_1, F_2\in\binom{\binom{X}{\mik 2}}{\mik d}$ with $\tau(F_1)=\tau(F_2)$.
\end{enumerate}
\end{defn}

We have the following corollary.

\begin{cor} \label{cor_from_Ramsey}
Let $n,d,r,k$ be positive integers with
\begin{equation} \label{e5.9}
n\geqslant \mathrm{R}\Big(2d,r+2d-1,2^{kt_d}\Big),
\end{equation}
where $t_d$ denotes the number of all possible types of elements of ${{\mathbb{N} \choose {\leqslant 2}} \choose {\leqslant d}}$. Then, for any collection $\boldsymbol{\lambda}=\big\langle \lambda_F\colon F\in \binom{\binom{[n]}{\mik 2}}{\mik d} \big\rangle$ of elements of $\mathbb{Z}_{2^k}$, there exists $X\in\binom{[n]}{r}$ such that $\boldsymbol{\lambda}$ is canonical in $X$.
\end{cor}
\begin{proof}
It follows immediately by Theorem \ref{thm-Ramsey-main}.
\end{proof}

The following lemma shows that integer polynomials whose coefficients are canonical in a set $X$ must have lower degree when restricted to appropriate $\mathrm{HJ}$-subspaces.

\begin{lem}[Degree-lowering on $\mathrm{HJ}$-subspaces] \label{lem_reduce_degree}
Let $n,k$ be positive integers, and let $X$ be a nonempty subset of $[n]$. Let $V$ be a block $\mathrm{HJ}$-subspace of $\mathbb{F}_2^{\binom{[n]}{\mik 2}}$, set $m:=\dim(V)$, and let $e_V\colon \mathbb{F}_2^{\binom{[m]}{\mik 2}}\to V$ be its associated embedding. (See Definition \ref{d2.2}.) Also let $Q\colon \mathbb{F}_2^{\binom{[n]}{\mik 2}} \to \mathbb{Z}_{2^k}$ be
a $k$-integer polynomial, set $d:=\mathrm{deg}(Q)$, and let $\boldsymbol{\lambda}=\big\langle \lambda_F\colon F\in
\binom{\binom{[n]}{\mik 2}}{\mik d}\big\rangle$ denote the coefficients of $Q$ written in the form \eqref{eq3.01}. Assume that $d\meg 1$ and that the following hold true.
\begin{enumerate}
\item[(i)] The collection $\boldsymbol{\lambda}$ is canonical in $X$.
\item[(ii)] Setting $\mathrm{wild}(V):=(I_i)_{i=1}^m$, then, for any $i\in [m]$, the set $I_i$ is a subset of $X$ with cardinality $(d+1)!\, 2^k$.
\end{enumerate}
Then, $\mathrm{deg}(Q\circ e_V)<\mathrm{deg}(Q)$.
\end{lem}

\begin{proof}
As in the proof of Lemma \ref{l5.5}, set $\mathcal{I}_n := \binom{[n]}{\leqslant 2}$ and $\mathcal{I}_m := \binom{[m]}{\leqslant 2}$; also let $\alpha'$ be as in \eqref{e5.4}, and let $\boldsymbol{\lambda}'=\big\langle \lambda'_T\colon T\in
\binom{\Ical_m}{\mik d}\big\rangle$ be as in \eqref{e5.5}. Recall that, for any $y\in \Ftwo{\Ical_m}$,
\begin{equation} \label{e5.10}
Q\circ e_V (y) = \alpha' + \sum_{T\in\binom{\mathcal{I}_m}{\mik d}} \lambda'_T \prod_{q\in T}y(q) \mod 2^k.
\end{equation}
Thus, it is enough to show that $\lambda'_T=0\mod 2^k$ for every $T\in\binom{\mathcal{I}_m}{d}$.

So, fix $T\in\binom{\mathcal{I}_m}{d}$, and let $\Hcal_T$ be as in \eqref{e5.6}. Then, for any $F\in\mathcal{H}_T$, we necessarily have  $|F|=d$ and $F\subseteq \bigcup_{q\in \Ical_m} \mathrm{var}(V)_q$, that yields that $\cup F\subseteq \bigcup_{i=1}^m I_i\subseteq X$.
Since $\boldsymbol{\lambda}$ is canonical in~$X$, we obtain that
\begin{equation} \label{eq3.02}
\lambda_{F_1} = \lambda_{F_2}
\end{equation}
for every pair $F_1,F_2\in\mathcal{H}_T$ with $\mathrm{\tau}(F_1) = \mathrm{\tau}(F_2)$. Let $F\in \mathcal{H}_T$ and recall that $|F|=d$; we define
\begin{enumerate}
\item[$\bullet$] the \emph{distributed local position $\mathrm{dlp}(F)=\big(\mathrm{dlp}(F)_1,\dots,\mathrm{dlp}(F)_m\big)$ of $F$},
\item[$\bullet$] the \emph{distributed signature $\mathrm{dsg}(F)=\big(\mathrm{dsg}(F)_1,\dots,\mathrm{dsg}(F)_m\big)$ of $F$}, and
\item[$\bullet$] the \emph{distributed type $\mathrm{dt}(F)=\big(\mathrm{dt}(F)_1,\dots,\mathrm{dt}(F)_m\big)$ of $F$},
\end{enumerate}
as follows. First, for every $i\in [m]$, set
\begin{equation} \label{e5.12}
\mathrm{dlp}(F)_i := I_i\cap (\cup F) \ \ \ \text{ and } \ \ \ \mathrm{dsg}(F)_i := |\mathrm{dlp}(F)_i|.
\end{equation}
Next, write $F$ in increasing lexicographical order as $(p_j)_{j=1}^d$. For every $i\in [m]$, write $\mathrm{dlp}(F)_i$ in increasing order as $\big\{u^i_1<\dots <u^i_{\mathrm{dsg}(F)_i}\big\}$, and set
\begin{equation} \label{e5.13}
\mathrm{dt}(F)_i := \big(\{ r\in [\mathrm{dsg}(F)_i] \colon u^i_r \in p_j \}\big)_{j=1}^d.
\end{equation}
Finally, let $\mathrm{DT} := \{\mathrm{dt}(F) \colon F\in \Hcal_T\}$ denote the set of all distributed types of members of $\Hcal_T$. Notice that $\mathrm{dlp}(F)$ and $\mathrm{dt}(F)$ uniquely determine $F$. This yields that, for every $F\in\mathcal{H}_T$, the sets
$\{ F'\in \Hcal_T\colon \mathrm{dt}(F')=\mathrm{dt}(F)\}$ and $\big\{ A\subseteq \bigcup_{i=1}^m I_i \colon |A\cap I_i |= \mathrm{dsg}(F)_i \text{ for all } i\in [m]\big\}$ have equal cardinalities. On the other hand, since $|T|=d$, for every $F\in\mathcal{H}_T$ and every $i\in [m]$, there is at most one $p\in F$ with $|I_i\cap p|=2$ and, therefore, $\mathrm{dsg}(F)_i\leqslant d+1$. Hence, by part (ii) of the lemma, for every $F\in\mathcal{H}_T$, the cardinality of the set $\big\{ A\subseteq \bigcup_{i=1}^m I_i \colon |A\cap I_i |= \mathrm{dsg}(F)_i \text{ for all } i\in [m]\big\}$ is a multiple of $2^k$. This yields that, for any distributed type $\tau\in\mathrm{DT}$,
\begin{equation} \label{eq3.03}
\big|\big\{ F\in\mathcal{H}_T \colon \mathrm{dt}(F) =\tau\big\}\big| = 0 \mod 2^k.
\end{equation}
Finally, note that if $F_1,F_2\in\mathcal{H}_T$ with $\mathrm{dt}(F_1) = \mathrm{dt}(F_2) $, then $\mathrm{\tau}(F_1) = \mathrm{\tau}(F_2)$ and, consequently, by \eqref{eq3.02},
\begin{equation} \label{eq3.04}
\lambda_{F_1} = \lambda_{F_2}.
\end{equation}
We conclude that
\begin{equation} \label{e5.16}
\lambda'_T =\sum_{F\in\mathcal{H}_T}\lambda_F =
\sum_{\tau\in\mathrm{DT}}\sum_{\substack{F\in \mathcal{H}_T\\ \mathrm{dt}(F)=\tau}}
\lambda_F\stackrel{\eqref{eq3.03},\eqref{eq3.04}}{=}0\mod 2^k. \qedhere
\end{equation}
\end{proof}

The following lemma is the last step of the proof of Theorem \ref{thm_part_int_pol}.

\begin{lem} \label{lem_ind_step}
Let $\eta>0$, and let $m,k,d$ be positive integers. Then there exists a positive integer $n_1 = n_1(\eta, m,k,d)$ with the following property. Let $n\geqslant n_1$ be an integer, let $Q\colon \mathbb{F}_2^{\binom{[n]}{\mik 2}}\to \mathbb{Z}_{2^k}$ be a
$k$-integer polynomial of degree $d$. Then there exists a collection $\mathcal{V}$ of pairwise disjoint $m$-dimensional block\, $\mathrm{HJ}$-subspaces of\, $\mathbb{F}_2^{\binom{[n]}{\leqslant2}}$ such that
\begin{enumerate}
\item[(i)] $\prob\Big[ \mathbb{F}_2^{\binom{[n]}{\mik 2}}\setminus (\cup\mathcal{V})\Big] \mik \eta$, and
\item[(ii)] for every $V\in\mathcal{V}$, we have $\mathrm{deg}(Q\circ e_V)<\mathrm{deg}(Q)$.
\end{enumerate}
\end{lem}

\begin{proof}
We start by introducing some numerical parameters. First, set
\begin{gather}
\label{e5.17} \eta_1 :=
2^{\frac{1}{2}\big(m(m+1)-2^k(d+1)!m\big(2^k(d+1)!m+1\big)\big)}, \\
\label{e5.18} \ell:=
\Big\lfloor \frac{\log \eta}{\log(1-\eta_1)} \Big\rfloor+1.
\end{gather}
Notice, by the choice of $\ell$ in \eqref{e5.18}, we have
\begin{equation} \label{eq3.05}
(1-\eta_1)^\ell <\eta.
\end{equation}
We also set
\begin{gather}
\label{e5.20} r: =2^k\,(d+1)!\,m\,\ell, \\
\label{e5.21} n_1:= \mathrm{R}\Big(2d,r+2d-1,2^{k t_d}\Big);
\end{gather}
here, as in Corollary \ref{cor_from_Ramsey}, $t_d$ denotes the number of all possible types of elements of ${{\mathbb{N} \choose {\leqslant 2}} \choose {\leqslant d}}$. We will show that $n_1$ is as desired.

To this end, let $n, Q$ be as in the statement of the lemma, and let $\boldsymbol{\lambda}=\big\langle \lambda_F\colon F\in
\binom{\binom{[n]}{\mik 2}}{\mik d}\big\rangle$ denote the coefficients of $Q$ written in the form \eqref{eq3.01}. By Corollary
\ref{cor_from_Ramsey}, there exists $X\in\binom{[n]}{r}$ such that $\boldsymbol{\lambda}$ is canonical in $X$. We are going to define several combinatorial objects related to the set $X$ that will enable us to introduce the desired collection $\mathcal{V}$.

\subsection*{Step 1}

Let $(X_j)_{j=1}^\ell$ be the unique finite sequence of successive subsets of $X$ that satisfies $|X_j|=2^k\,(d+1)!\,m$ for every $j\in [\ell]$; moreover, for every $j\in[\ell]$, let $(I^j_i)_{i=1}^m$ be the unique finite sequence of successive subsets of $X_j$ such that $|I^j_i|=2^k\, (d+1)!$ for every $i\in [m]$. Thus,
\begin{equation} \label{eq3.06}
X= X_1\cup\cdots\cup X_\ell \ \ \ \text{ and } \ \ \ X_j=I^j_1\cup\cdots\cup I^j_m \text{ for all } j\in [\ell].
\end{equation}
Moreover, for every $j\in[\ell]$ and every $q\in \binom{[m]}{\mik 2}$, set
\begin{equation} \label{e5.23}
\mathrm{var}^j_q := \bigg\{p\in\binom{\bigcup_{i\in q}I^j_i}{\leqslant2}\colon p\cap I^j_i\neq\emptyset \text{ for all } i\in q\bigg\}.
\end{equation}

\subsection*{Step 2}

Next, for every $j\in [\ell]$, set
\begin{gather}
\label{e5.24} \mathcal{X}_j^\mathrm{con} :=
\bigg\{ x\in\mathbb{F}_2^{\binom{X_j}{\mik 2}} \colon x \text{ is constant on }\mathrm{var}_q^j
\text{ for all } q\in \binom{[m]}{\mik 2}\bigg\}, \\
\label{e5.25} \mathcal{X}_j^\mathrm{ncon} := \mathbb{F}_2^{\binom{X_j}{\mik 2}}\setminus \mathcal{X}_j^\mathrm{con};
\end{gather}
notice that
\begin{equation} \label{eq3.07}
\frac{|\mathcal{X}_j^\mathrm{con}|}{\Big|\mathbb{F}_2^{\binom{X_j}{\mik 2}}\Big|} =\eta_1
\ \ \  \text{ and } \ \ \
\frac{|\mathcal{X}_j^\mathrm{ncon}|}{\Big|\mathbb{F}_2^{\binom{X_j}{\mik 2}}\Big|} =1-\eta_1.
\end{equation}

\subsection*{Step 3}

Now, for every $z\in \mathbb{F}_2^{\binom{[n]}{\mik 2}\setminus\binom{X_1}{\mik 2}}$, set
\begin{equation} \label{e5.27}
V^1_z := \big\{ x\cup z \colon x\in \mathcal{X}_1^\mathrm{con}\big\},
\end{equation}
and define
\begin{equation} \label{e5.28}
\mathcal{V}_1 := \bigg\{ V^1_z \colon z\in \mathbb{F}_2^{\binom{[n]}{\mik 2}\setminus\binom{X_1}{\mik 2}}\bigg\}.
\end{equation}
Respectively, for every $j\in[\ell]$ with $j\geqslant2$, every $y_1\in \mathcal{X}_1^\mathrm{ncon},\dots,y_{j-1} \in
\mathcal{X}_{j-1}^\mathrm{ncon}$ and every $z\in \mathbb{F}_2^{\binom{[n]}{\mik 2}\setminus
\big( \binom{X_1}{\mik 2}\cup \cdots\cup \binom{X_j}{\mik 2}\big)}$, set
\begin{equation} \label{e5.29}
V^j_{y_1,\dots,y_{j-1},z} := \bigg\{ x\cup y_1\cup \cdots \cup y_{j-1}\cup z \colon  x\in \mathcal{X}_j^\mathrm{con}\bigg\},
\end{equation}
and define
\begin{equation} \label{e5.30}
\mathcal{V}_j := \bigg \{ V^j_{y_1,\dots,y_{j-1},z} \colon
y_1\in \mathcal{X}_1^\mathrm{ncon},\dots ,y_{j-1} \in \mathcal{X}_{j-1}^\mathrm{ncon},
z\in \mathbb{F}_2^{\binom{[n]}{\mik 2}\setminus \big( \binom{X_1}{\mik 2}\cup \cdots\cup \binom{X_j}{\mik 2}\big)} \bigg\}.
\end{equation}

\subsection*{Step 4: properties}

The following properties are guaranteed by the above construction.
\begin{enumerate}
\item[(P1)] For every $j\in [\ell]$, the elements of $\mathcal{V}_j$ are $m$-dimensional block $\mathrm{HJ}$-subspaces.
\item[(P2)] For every $j\in [\ell]$, the subspaces in $\mathcal{V}_j$ are pairwise disjoint.
\item[(P3)] For every distinct $j,j'\in [\ell]$, we have that $(\cup \mathcal{V}_j)\cap(\cup \mathcal{V}_{j'})=\emptyset$.
\item[(P4)] \label{new-property} For every $j\in [\ell]$, by \eqref{eq3.07}, we have $\prob\big[ (\cup\mathcal{V}_j)\big] =\eta_1(1-\eta_1)^{j-1}$.
\item[(P5)] For every $j\in [\ell]$ and every $V\in\mathcal{V}_j$, we have that $\mathrm{wild}(V)=(I_i^j)_{i=1}^m$; also, recall that $I_i^j$ is a subset of $X$ of cardinality $2^k\, (d+1)!$ for every $i\in [m]$.
\end{enumerate}
Setting $\mathcal{V}:= \mathcal{V}_1\cup\cdots\cup \mathcal{V}_\ell$, the proof is completed by Lemma \ref{lem_reduce_degree} and observing that
\begin{equation} \label{e5.31}
\prob\Big[ \mathbb{F}_2^{\binom{[n]}{\mik 2}}\setminus (\cup\mathcal{V})\Big]
\stackrel{(\hyperref[new-property]{\mathrm{P4}})}{=}  (1-\eta_1)^\ell\stackrel{\eqref{eq3.05}}{\mik }\eta. \qedhere
\end{equation}
\end{proof}

We are finally in a position to complete the proof of Theorem \ref{thm_part_int_pol}.

\begin{proof}[Completion of the proof of Theorem \ref{thm_part_int_pol}]
It follows by repeated applications of Lemma \ref{lem_ind_step}.
\end{proof}

\section{Proof of Theorem \ref*{thm:MainCliques}} \label{sec6}

As in \eqref{e1.6}, let $\Ccal^\circ$ denote the set of nonempty cliques with all possible self-loops. Also recall that, for any positive integer $n$, by $\Delta_n^\circ(\Ccal^\circ)$ we denote the largest density of a $\Ccal^\circ$-$\mathrm{HJ}$-code $\Gcal\subseteq \Ftwo{\binom{[n]}{\mik 2}}$. (See Definition \ref{d1.10}.) By Fact \ref{f1.11}, the sequence $\big(\Delta_n^\circ(\Ccal^\circ)\big)$ is non-increasing.

\begin{lem} \label{lem_decreasing_stong_codes}
Let $n$ be a positive integer, and let\, $\mathcal{G}\subseteq\mathbb{F}_2^{\binom{[n]}{\leqslant2}}$ be a $\Ccal^\circ$-$\mathrm{HJ}$-code. Also let $V$ be a $\mathrm{HJ}$-subspace of\, $\mathbb{F}_2^{\binom{[n]}{\leqslant2}}$. Then, $\prob\big[ \Gcal\, \big|\, V\big] \mik \Delta^\circ_{\mathrm{dim}(V)}(\Ccal^\circ)$.
\end{lem}

\begin{proof}
Set $m:=\dim(V)$, let $e_V\colon \mathbb{F}_2^{\binom{[m]}{\mik 2}}\to V$ be its associated embedding (see Definition \ref{d2.2}), and set
$\Gcal':= e_V^{-1}(\Gcal)$. Then observe that $\mathcal{G}'$ is also a $\Ccal^\circ$-$\mathrm{HJ}$-code, and so,
\begin{equation} \label{e6.1}
\Delta^\circ_{m}(\Ccal^\circ) \meg \prob[\Gcal'\big] = \prob\big[ \Gcal\, \big|\, V\big]. \qedhere
\end{equation}
\end{proof}

The following lemma is the last ingredient of the proof of Theorem \ref{thm:MainCliques}.

\begin{lem} \label{lem:poly_cor}
Let $m,d $ be positive integers, and let $\eta > 0$. Also let $n \geqslant n_2(\eta,m ,d)$ be an integer,~where $n_2(\eta,m ,d)$ is as in Theorem \ref{thm_part_nclas_pol}, and assume that $\Delta_m^\circ(\Ccal^\circ)-\Delta_n^\circ(\Ccal^\circ)\leqslant\eta$. Fix a $\Ccal^\circ$-$\mathrm{HJ}$-code $\Gcal\subseteq\mathbb{F}_2^{\binom{[n]}{\mik 2}}$ with $\prob[\Gcal]=\Delta_n^\circ(\Ccal^\circ)$---that is, the code $\Gcal$ is extremal in the sense of Definition \ref{d1.10}---and a nonclassical polynomial $P\colon \mathbb{F}_2^{\binom{[n]}{\mik 2}}\to\mathbb{T}$ of degree at most~$d$. Then,
\begin{equation} \label{e6.2}
\Big|\mathbb{E} \Big[ \big( \mathbbm{1}_{\Gcal}-\prob[\Gcal] \big)\, \exp(2\pi i P)\Big]\Big| \leqslant 4\eta.
\end{equation}
\end{lem}

\begin{proof}
By adding a constant to $P$ if necessary, we can assume that $\mathbb{E} \big[ \big( \mathbbm{1}_{\Gcal}-\prob[\Gcal] \big)\, \exp(2\pi i P)\big]$ is a nonnegative real number.

Since $n\geqslant n_2(\eta,m,d)$, by Theorem \ref{thm_part_nclas_pol}, there is a collection $\mathcal{V}$ of pairwise disjoint $m\text{-dimensional}$ block $\mathrm{HJ}$-subspaces of $\mathbb{F}_2^{\binom{[n]}{\mik 2}}$ such that
\begin{enumerate}[label=(\roman*)]
\item\label{part-i} $\prob\Big[\mathbb{F}_2^{\binom{[n]}{\mik 2}}\setminus (\cup\mathcal{V})\Big] \mik \eta$, and
\item\label{part-ii} for every $V\in\mathcal{V}$, the polynomial $P$ is constant on $V$.
\end{enumerate}
Fix $V \in \mathcal{V}$. By Lemma \ref{lem_decreasing_stong_codes}, we have $\prob\big[\Gcal \, \big| \, V\big] \mik \Delta_m^\circ(\Ccal^\circ)$ and, consequently,
\begin{equation} \label{eq4.002*}
\mathbb{E}\big[\mathbbm{1}_{\Gcal} - \prob[\Gcal]\, \big| \, V\big] \mik \Delta_m^\circ(\Ccal^\circ) - \prob\big[\Gcal\big] = \Delta_m^\circ(\Ccal^\circ) - \Delta_n^\circ(\Ccal^\circ) \leqslant \eta.
\end{equation}
Set $f := 1+\cos(2\pi P)$. By \ref{part-ii}, the function $f$ is constant on $V$; denote by $\lambda_V$ this constant value. Since $0 \leqslant \lambda_V \leqslant 2$, we have
\begin{equation} \label{eq4.003*}
\mathbb{E}\big[ \big(\mathbbm{1}_{\Gcal} - \prob[\Gcal]\big)f \, \big| \, V\big] = \lambda_V\,\mathbb{E}\big[\mathbbm{1}_{\Gcal} - \prob[\Gcal]\, \big| \, V\big] \stackrel{\eqref{eq4.002*}}{\leqslant} 2 \eta.
\end{equation}
Using the fact that the elements of $\mathcal{V}$ are pairwise disjoint, we deduce that
\begin{equation} \label{eq4.004*}
\mathbb{E}\big[ \big(\mathbbm{1}_{\Gcal} - \prob[\Gcal]\big)f \, \big| \, (\cup\mathcal{V})\big] \leqslant 2\eta.
\end{equation}
Hence,
\begin{multline} \label{eq4.005*}
\mathbb{E}\big[\big(\mathbbm{1}_{\Gcal} - \prob[\Gcal]\big)f\big]  \\
\leqslant \mathbb{E}\big[\big(\mathbbm{1}_{\Gcal} - \prob[\Gcal]\big)f \, \big| \, (\cup\mathcal{V})\big]  + \mathbb{E}\Big[\big(\mathbbm{1}_{\Gcal} - \prob[\Gcal]\big)f \, \Big| \, \mathbb{F}_2^{\binom{[n]}{\mik 2}}\setminus (\cup\mathcal{V})\Big] \cdot \prob\Big[\mathbb{F}_2^{\binom{[n]}{\mik 2}}\setminus (\cup\mathcal{V})\Big] \\
\stackrel{\eqref{eq4.004*}, \, \ref{part-i}}{\mik} 2\eta + 2\eta = 4\eta.
\end{multline}
Since $\mathbb{E}\big[\mathbbm{1}_{\Gcal} - \prob[\Gcal]\big] = 0$, it follows that
\begin{equation} \label{eq4.006*}
\operatorname{Re}\Big(\mathbb{E} \big[ \big( \mathbbm{1}_{\Gcal}-\prob[\Gcal] \big)\, \exp(2\pi i P)\big]\Big) = \mathbb{E}\big[\big(\mathbbm{1}_{\Gcal} - \prob[\Gcal]\big)\,(f - 1)\big] \leqslant 4\eta,
\end{equation}
where, as usual, $\operatorname{Re}(z)$ denotes the real part of $z$. The result follows from our starting assumption that $\mathbb{E} \big[ \big( \mathbbm{1}_{\Gcal}-\prob[\Gcal] \big)\, \exp(2\pi i P)\big]$ is a nonnegative real number.
\end{proof}

We are ready to complete the proof of Theorem \ref{thm:MainCliques}.

\begin{proof}[Completion of the proof of Theorem \ref{thm:MainCliques}]
As we have already pointed out, by Fact \ref{f1.11}, the sequence $\big(\Delta_n^\circ(\Ccal^\circ)\big)$ is non-increasing and bounded, and consequently, it is convergent. Using this observation, the proof follows by the inverse theorem for the $U_d$-norm (Theorem~\ref{thm:inverse}) together with Lemma \ref{lem:poly_cor}.
\end{proof}

\begin{rem}
Let $\mathcal{P}$ be a property of subsets of sets of the form $\Ftwo{\binom{[n]}{\leqslant 2}}$, where $n \geqslant 2$ is an integer. We say that the property $\mathcal{P}$ is \textit{$\mathrm{HJ}$-hereditary} if it is preserved under taking preimages by $\mathrm{HJ}$-embeddings (see Subsection \ref{subsec2.2}). For instance, the property of being a $\Ccal^\circ$-$\mathrm{HJ}$-code and the property of being a $\Ccal^\circ$-code are both $\mathrm{HJ}$-hereditary. If $\mathcal{P}$ is a $\mathrm{HJ}$-hereditary property and $n \geqslant 2$ is an integer, then we say that a subset $\mathcal{G} \subseteq \Ftwo{\binom{[n]}{\leqslant 2}}$ is \textit{$\mathcal{P}$-extremal} if it satisfies $\mathcal{P}$ and has maximum density among subsets of $\Ftwo{\binom{[n]}{\leqslant 2}}$ satisfying $\mathcal{P}$. The proof of Theorem \ref{thm:MainCliques} immediately generalizes to show that $\mathcal{P}$-extremal sets are higher order uniform in the following sense: for every integer $d\geqslant 2$ and every $\varepsilon>0$, there exists a positive integer $n_0=n_0(d,\varepsilon)$ such that, for all $n\geqslant n_0$, if\, $\Gcal\subseteq \mathbb{F}_2^{\binom{[n]}{\leqslant 2}}$ is $\mathcal{P}$-extremal,~then
\begin{equation} \label{e6.14}
\big\|\mathbbm{1}_{\Gcal}-\prob[\Gcal]\big\|_{U_d} \mik \varepsilon.
\end{equation}
In particular, extremal $\Ccal^\circ$-codes are also higher order uniform. We also note that our approach extends to spaces of the form $\Ftwo{\binom{[n]}{\leqslant k}}$, where $k\geqslant 2$ is any fixed integer.
\end{rem}

\appendix

\section{ \ } \label{sec-appendix}

We start by recalling some definitions related to the polynomial Hales--Jewett theorem \cite{BL99}. For every triple $n,d,k$ of positive integers with $k\meg 2$, let $W(n,d,k)$ denote the set of all maps $w \colon [n]^d \to [k]$; we shall refer to the elements of $W(n,d,k)$ as \emph{polynomial words over $[k]$}. A \emph{polynomial variable word $v$ of\, $W(n,d,k)$} is a map $v\colon [n]^d \to [k]\cup\{x\}$, where $x$ is a symbol not belonging to $[k]$, such that $v^{-1}(\{x\}) = X^d$ for some nonempty subset $X$ of $[n]$. Given a polynomial variable word $v$ of $W(n,d,k)$ and $a \in [k]$, let $v(a)$ denote the element of $W(n,d,k)$ obtained by replacing every occurrence of the symbol $x$ in $v$ by $a$. A \emph{polynomial combinatorial line of $W(n,d,k)$} is a set of the form
$\big\{v(a)\colon a\in [k]\big\}$, where $v$ is a polynomial variable word of $W(n,d,k)$. We are ready to recall the density polynomial Hales--Jewett conjecture.

\begin{DPHJ-conjecture}[Bergelson \cite{Ber96}]
For every pair $d,k$ of positive integers with $k\meg 2$ and every\, $0<\delta\mik 1$, there exists a positive integer $\mathrm{DPHJ}(d,k,\delta)$ such that if $n\meg \mathrm{DPHJ}(d,k,\delta)$, then any subset of $W(n,d,k)$ with cardinality at least $\delta\, k^{n^d}$ contains a polynomial combinatorial line of\, $W(n,d,k)$.
\end{DPHJ-conjecture}

The case ``$d=1$" of the above conjecture is the famous density Hales--Jewett theorem due to Furstenberg--Katznelson \cite{FK91}.
However, even the first higher-dimensional case, ``$d=k=2$", is open and it is considered a major problem in density Ramsey theory \cite{Alon24,Gow09,DK16}.

Our goal in this Appendix is to show that the first unknown case of the density polynomial Hales--Jewett conjecture, ``$d=k=2$", is equivalent to an affirmative answer to Problem \ref{p1.8}. Specifically, we have the following proposition.

\begin{prop} \label{p-a.1}
The following are equivalent.
\begin{enumerate}
\item[(i)] For any $0<\delta\mik 1$, the positive integer\, $\mathrm{DPHJ}(2,2,\delta)$ exists.
\item[(ii)] Problem \ref{p1.8} has an affirmative answer; that is, if\, $\Gcal\subseteq \Ftwo{\binom{[n]}{\mik 2}}$ is a $\Ccal^\circ$-$\mathrm{HJ}$-code, then we have $\prob[\Gcal]=o_{n\to\infty}(1)$.
\end{enumerate}
\end{prop}

\begin{proof}
We first argue for the implication (i)$\Rightarrow$(ii); so, assume that, for any $0<\delta\mik 1$, the positive integer $\mathrm{DPHJ}(2,2,\delta)$ exists. It is enough to show that, for any $0<\delta\mik 1$ and any integer $n\geqslant \mathrm{DPHJ}(2,2,\delta)$, if $\mathcal{G}\subseteq \mathbb{F}_2^{\binom{[n]}{\mik 2}}$ satisfies $\prob[\Gcal]\meg \delta$, then $\mathcal{G}$ is not a $\Ccal^\circ$-$\mathrm{HJ}$-code. Indeed, set
\begin{equation} \label{e.a.1}
D := \bigg\{ x\in \mathbb{F}_2^{[n]^2} \colon \Big\langle x(\min e, \max e)\colon e\in \binom{[n]}{\mik 2} \Big\rangle \in\mathcal{G} \bigg\}.
\end{equation}
Observe that $|D|\meg \prob[\Gcal]\cdot 2^{n^2} \meg \delta\, 2^{n^2}$; also notice that $D$ contains a polynomial combinatorial line if and only if $\mathcal{G}$ is not a $\Ccal^\circ$-$\mathrm{HJ}$-code. Since $n\geqslant \mathrm{DPHJ}(2,2,\delta)$, it follows that $\mathcal{G}$ is not a $\Ccal^\circ$-$\mathrm{HJ}$-code, as desired.

We proceed to show that (ii)$\Rightarrow$(i). Let $\big(\Delta_n^\circ(\Ccal^\circ)\big)$ be the sequence defined in \eqref{e1.11}, and note that our assumption in this case is equivalent to saying that the sequence $\big(\Delta_n^\circ(\Ccal^\circ)\big)$ converges to zero. We will need a ``conditional concentration" estimate that originates from \cite{DKT16}; the version stated below is taken from \cite[Lemma 8.1]{DTV23}.

\begin{lem}\label{lem_regularity}
Let $0<\varepsilon \mik 1$, and let $\ell,m$ be positive integers such that
\begin{equation} \label{eq:6.01}
\ell\meg \frac{2^{m+1}}{\varepsilon^2}.
\end{equation}
Let $I$ be a nonempty finite set, and let $\prob$ denote the uniform probability measure on $\{0,1\}^I$. Also let $D_1,\dots,D_\ell$ be pairwise disjoint nonempty subsets of $I$ each with at most $m$ elements. If $A$ is any subset of $\{0,1\}^I$, then there exists $i_0 \in[\ell]$ such that, for every $x\in \{0,1\}^{D_{i_0}}$, setting\, $S_x := \big\{ y\cup x \colon y\in \{0,1\}^{I\setminus D_{i_0}}\big\}$, we have
\begin{equation} \label{e.a.3}
\big| \prob\big[A\, \big|\, S_x\big] - \prob[A]\big| \mik \varepsilon.
\end{equation}
\end{lem}
Now assume, towards a contradiction, that part (i)  does not hold true, that is, there exists $\delta>0$ such that, for every positive integer $n_0$, there exist an integer $n\geqslant n_0$ and a subset $A\subseteq \mathbb{F}_2^{[n]^2}$ with
$\prob[A]\geqslant\delta$ such that $A$ contains no polynomial combinatorial line.

Let $n_1$ be an arbitrary positive integer, and set
\begin{equation} \label{e.a.4}
m := n_1^2, \ \ \ \ell := \left\lceil\frac{2^{m+1}}{(\delta/2)^2}\right\rceil \ \ \ \text{ and } \ \ \ n_0:= n_1 \ell;
\end{equation}
moreover, for every $i\in [\ell]$, set
\begin{equation} \label{e.a.5}
I_i := \big\{(i-1)n_1+1,\dots, i n_1\big\} \ \ \ \text{ and } \ \ \  D_i := I_i\times I_i.
\end{equation}
By our assumption that part (i) does not hold true, we may select an integer $n\geqslant n_0$ and a set $A\subseteq \mathbb{F}_2^{[n]^2}$ with $\prob[A]\meg \delta$ that contains no polynomial combinatorial line. By Lemma \ref{lem_regularity}, there exists $i_0\in[\ell]$ such that, for every $x\in \mathbb{F}_2^{D_{i_0}}$,
\begin{equation}\label{eq:022}
\prob\big[A\, \big|\, S_x\big] \geqslant \frac{\delta}{2},
\end{equation}
where $S_x = \Big\{y\cup x \colon y\in \mathbb{F}_2^{I\setminus D_{i_0}}\Big\}$. Write $I_{i_0}$ in increasing order as $\{k_1 <\cdots < k_{n_1}\}$, and set
\begin{equation} \label{e.a.7}
\mathcal{S} := \Big\{ x \in \mathbb{F}_2^{D_{i_0}} \colon x(k_i , k_j) = x(k_j , k_i)  \text{ for all } i,j\in[n_1] \Big\}
\ \ \ \text{ and } \ \ \ \mathcal{Y} := \mathbb{F}_2^{[n]^2 \setminus D_{i_0}}.
\end{equation}
Moreover, setting $V^{\mathcal{S}}_y := \big\{ y \cup x \colon x \in \mathcal{S}\big\}$ for every $y\in\mathcal{Y}$, by \eqref{eq:022}, we have
\begin{equation} \label{e.a.8}
\frac{\delta}{2} \mik  \underset{x\in\mathcal{S}}{\ave} \Big[ \prob\big[ A\, \big|\, S_x\big]\Big] =
 \underset{y\in\mathcal{Y}}{\ave}\Big[ \prob\big[ A\, \big|\, V^{\mathcal{S}}_y\big]\Big].
\end{equation}
Thus, there exists $y_0\in\mathcal{Y}$ such that $\prob\big[ A\, \big|\, V_{y_0}^{\mathcal{S}}\big]\meg \delta/2$; note that, since $A$ contains no polynomial combinatorial line, the set $A\cap V_{y_0}^{\mathcal{S}}$ also does not contain a polynomial combinatorial line.

For every $z\in \Ftwo{\binom{[n_1]}{\mik 2}}$, define $x_z\in \mathbb{F}_2^{D_{i_0}}$ by setting $x_z(k_i,k_j) = z(\{i,j\})$ for all $i,j\in [n_1]$. Moreover, for every $z\in \Ftwo{\binom{[n_1]}{\mik 2}}$, set $u_z := y_0 \cup x_z$. Clearly, $x_z\in \mathcal{S}$, and $u_z \in V_{y_0}^{\mathcal{S}}$ for every $z\in \Ftwo{\binom{[n_1]}{\mik 2}}$. Finally, define
\begin{equation} \label{e.a.10}
\Gcal := \Big\{ z\in \Ftwo{\binom{[n_1]}{\mik 2}} \colon u_z \in A \Big\};
\end{equation}
Then, $\prob[\Gcal] = \prob\big[A\, \big|\, V_{y_0}^{\mathcal{S}}\big] \meg \delta/2$ and, since $A$ contains no polynomial combinatorial line, the family $\Gcal$ is a $\Ccal^\circ$-$\mathrm{HJ}$-code. Therefore, $\Delta_{n_1}^{\circ}(\Ccal^\circ) \geqslant \delta/2$ and, since $n_1$ was an arbitrary positive integer, this contradicts our assumption that the sequence $\big( \Delta_n^{\circ}(\Ccal^\circ)\big)$ converges to zero. The proof of the proposition is thus completed.
\end{proof}

\subsection*{Acknowledgments}

The authors would like to thank the anonymous referee for numerous comments, remarks and suggestions that helped us improve the exposition.

The research was supported by the Hellenic Foundation for Research and Innovation (H.F.R.I.) under the “2nd Call for H.F.R.I. Research Projects to support Faculty Members \& Researchers” (Project Number: HFRI-FM20-02717). No\'e de Rancourt additionally acknowledges support from the Labex CEMPI (ANR-11-LABX-0007-01) and the CDP C2EMPI, together with the French State under the France-2030 programme, the University of Lille, the Initiative of Excellence of the University of Lille, the European Metropolis of Lille for their funding and support of the R-CDP-24-004-C2EMPI project.


\begin{thebibliography}{99}

\bibitem[Alon24]{Alon24}
N. Alon,
\emph{Graph-codes},
European J. Combin. 116 (2024), Article ID 103880, 7 p.

\bibitem[Ber96]{Ber96}
V. Bergelson,
\emph{Ergodic Ramsey theory---an update},
in ``Ergodic Theory of $\mathbb{Z}^d$-Actions"\!, London Mathematical Society Lecture Note Series, Vol. 228, Cambridge University Press, 1996, 1--61.

\bibitem[BL99]{BL99}
V. Bergelson and A. Leibman,
\emph{Set-polynomials and polynomial extension of the Hales--Jewett theorem},
Ann. Math. 150 (1999), 33--75.

\bibitem[CGW88]{CGW88}
F. R. K. Chung, R. L. Graham and R. M. Wilson,
\emph{Quasi-random graphs},
Proc. Natl. Acad. Sci. USA 85 (1988), 969--970.

\bibitem[CGW89]{CGW89}
F. R. K. Chung, R. L. Graham and R. M. Wilson,
\textit{Quasi-random graphs},
Combinatorica 9 (1989), 345--362.

\bibitem[DK16]{DK16}
P. Dodos and V. Kanellopoulos,
\emph{Ramsey Theory for Product Spaces},
Mathematical Surveys and Monographs, Vol. 212, American Mathematical Society, 2016.

\bibitem[DKT16]{DKT16}
P. Dodos, V. Kanellopoulos and K. Tyros,
\textit{A concentration inequality for product spaces},
J. Funct. Anal. 270 (2016), 609--620.

\bibitem[DTV23]{DTV23}
P. Dodos, K. Tyros and P. Valettas,
\emph{Concentration estimates for functions of finite high-dimensional random arrays},
Random Struct. Algorithms 63 (2023), 997--1053.

\bibitem[ES35]{ES35}
P. Erd\H{o}s and G. Szekeres,
\emph{A combinatorial problem in geometry},
Compositio Math. 2 (1935), 463--470.

\bibitem[FK91]{FK91}
H. Furstenberg and Y. Katznelson,
\emph{A density version of the Hales--Jewett theorem},
J. Anal. Math. 57 (1991), 64--119.

\bibitem[Gow01]{Gow01}
W. T. Gowers,
\emph{A new proof of Szemer\'{e}di's theorem},
Geom. Funct. Anal. 11 (2001), 465--588.

\bibitem[Gow09]{Gow09}
W. T. Gowers,
\textit{The first unknown case of polynomial DHJ},
blog post (2009), available at
\url{https://gowers.wordpress.com/2009/11/14/the-first-unknown-case-of-polynomial-dhj/}.

\bibitem[GGMT25]{GGMT25}
W. T. Gowers, B. Green, F. Manners and T. Tao,
\emph{On a conjecture of Marton},
Ann. Math. 201 (2025), 515--549.

\bibitem[HHL19]{HHL19}
H. Hatami, P. Hatami, S. Lovett,
\emph{Higher-order Fourier analysis and applications},
Found. Trends Theor. Comput. Sci. 13 (2019), 247--448.

\bibitem[O'Don14]{O'Don14}
Ryan O’Donnell,
\emph{Analysis of Boolean Functions},
Cambridge University Press, 2014.

\bibitem[Ra30]{Ra30}
F. P. Ramsey,
\emph{On a problem of formal logic},
Proc. London Math. Soc. 30 (1930), 264--286.

\bibitem[R\H{o}dl15]{Ro15}
V. R\H{o}dl,
\emph{Quasi-randomness and the regularity method in hypergraphs},
in ``Proceedings of the International Congress of Mathematicians" Vol. I, 571--599, 2015.

\bibitem[S\'{o}s13]{Sos13}
V. T. S\'{o}s,
\emph{Induced subgraphs and Ramsey colorings, 2013},
presented at the 16th International Conference on Random Structures and Algorithms, available at
\url{https://web.archive.org/web/20150910071523/http://rsa2013.amu.edu.pl/abstracts/Sos.Vera.pdf}.

\bibitem[Tao08]{Tao08}
T. Tao,
\emph{Structure and Randomness: Pages from Year One of a Mathematical Blog},
American Mathematical Society, Providence, RI, 2008.

\bibitem[TV06]{TV06}
T. Tao and V. Vu,
\textit{Additive Combinatorics},
Cambridge Studies in Advanced Mathematics, Vol. 105, Cambridge University Press, 2006.

\bibitem[TZ12]{TZ12}
T. Tao and T. Ziegler,
\emph{The inverse conjecture for the Gowers norm over finite fields in low characteristic},
Ann. Comb. 16 (2012), 121--188.

\bibitem[Tho87]{Tho87}
A. Thomason,
\emph{Pseudo-random graphs},
Ann. Discrete Math. 33 (1987), 307--331.


\end{thebibliography}
\end{document}